\documentclass[10pt,reqno]{amsproc}

\usepackage{amsmath,amsthm,amsfonts,amscd,amssymb,amstext,bbm}
\usepackage{xparse,mathrsfs,mathtools,todonotes,listings}
\usepackage{vmargin}
\usepackage{relsize}
\usepackage[colorlinks,citecolor=blue]{hyperref}
\newtheorem{theorem}{Theorem}

\newtheorem{remark}[theorem]{Remark}

\newtheorem{lemma}[theorem]{Lemma}
\newtheorem{proposition}[theorem]{Proposition}
\newtheorem{corollary}{Corollary}

\numberwithin{equation}{section}
\numberwithin{theorem}{section}
\usepackage{suffix}
\usepackage{mathtools}
\usepackage{xparse}

\DeclarePairedDelimiterX\MeijerM[3]{\lparen}{\rparen}
{\begin{smallmatrix}#1 \\ #2\end{smallmatrix}\delimsize\vert\,#3}
\newcommand\MeijerG[8][]{%
  G^{\,#2,#3}_{#4,#5}\MeijerM[#1]{#6}{#7}{#8}}

\WithSuffix\newcommand\MeijerG*[7]{%
  G^{\,#1,#2}_{#3,#4}\MeijerM*{#5}{#6}{#7}}

\title{On the number of real eigenvalues of a product of truncated orthogonal random matrices}

\author{Alex Little}
\address[A. Little]{\newline School of Mathematics, University of Bristol, BS8 1UG, United Kingdom. \textit{al17344@bristol.ac.uk}}

\author{Francesco Mezzadri}
\address[F. Mezzadri]{\newline School of Mathematics, University of Bristol, BS8 1UG, United Kingdom. \textit{f.mezzadri@bristol.ac.uk}}

\author{Nick Simm}
\address[N. Simm]{\newline Department of Mathematics, University of Sussex, Falmer Campus, Brighton, BN1 9RH, United Kingdom. \textit{n.j.simm@sussex.ac.uk}}
\date{}

\begin{document}

\maketitle

\begin{abstract}
Let $O$ be chosen uniformly at random from the group of $(N+L) \times (N+L)$ orthogonal matrices. Denote by $\tilde{O}$ the upper-left $N \times N$ corner of $O$, which we refer to as a \textit{truncation} of $O$. In this paper we prove two conjectures of Forrester, Ipsen and Kumar \cite{FIK20} on the number of real eigenvalues $N^{(m)}_{\mathbb{R}}$ of the product matrix $\tilde{O}_{1}\ldots \tilde{O}_{m}$, where the matrices $\{\tilde{O}_{j}\}_{j=1}^{m}$ are independent copies of $\tilde{O}$. When $L$ grows in proportion to $N$, we prove that
\begin{equation*}
\mathbb{E}(N^{(m)}_{\mathbb{R}}) = \sqrt{\frac{2m L}{\pi}}\,\mathrm{arctanh}\left(\sqrt{\frac{N}{N+L}}\right) + O(1), \qquad N \to \infty. 
\end{equation*}
We also prove the conjectured form of the limiting real eigenvalue distribution of the product matrix. Finally, we consider the opposite regime where $L$ is fixed with respect to $N$, known as the \textit{regime of weak non-orthogonality}. In this case each matrix in the product is very close to an orthogonal matrix. We show that $\mathbb{E}(N^{(m)}_{\mathbb{R}}) \sim c_{L,m}\,\log(N)$ as $N \to \infty$ and compute the constant $c_{L,m}$ explicitly. These results generalise the known results in the one matrix case due to Khoruzhenko, Sommers and \.{Z}yczkowski \cite{KSZ10}.
\end{abstract}

\section{Introduction and main results}

Research on products of random matrices started in 1960 with the work of Furstenberg and Kesten~\cite{FK60}. Early investigations concerned the properties and computation of the Lyapunov exponents of products of random matrices when the dimension $N$ is fixed and the number of factors $m$ grows to infinity, see e.g.~\cite{BL85,CPV93} and references therein.  More recently there were several new developments related to the opposite regime, namely when the number of factors $m$ is finite and the matrix dimension $N$ tends to infinity, or is also finite \cite{AI15}. Early progress in this direction came with the work of Burda et al.~\cite{BJW10}, who
computed the spectral density for a product consisting of independent non-Hermitian matrices whose elements are i.i.d. standard
complex normal random variables, the \emph{complex Ginibre ensemble,} in the limit $N \to \infty$ and for finite $m$.  Subsequently, Akemann and Burda~\cite{AB12} discovered that the determinantal structure of the eigenvalue point process known to hold for a single complex Ginibre matrix continues to hold for products of such matrices. However, the correlation kernel that arises for products is more complicated and has to be expressed in terms of Meijer G-functions. The determinantal machinery allowed them to analyse the point process of eigenvalues in various microscopic regimes for fixed $m$. A similar theory has been developed to analyse products of truncated unitary random matrices \cite{ABKN14, KNTK16, LW16}. Recently these developments have been extended to study the question of double scaling limits as both $m$ and $N$ tend to infinity simultaneously \cite{ABK20, LWW18, LW19}.

The focus of this paper will be on products of real random matrices. One of the most well-studied examples consists of matrices whose elements are real i.i.d. standard normal random variables, known as the \emph{real Ginibre ensemble}. A distinguishing feature of the real case is that there is a non-zero probability of finding purely real eigenvalues and therefore the point process of eigenvalues consists of two correlated components of purely real and complex conjugated points. This makes the study of real random matrices considerably more challenging than their complex counterparts. The real Ginibre ensemble was introduced in 1965 by Ginibre~\cite{Gin65}, but after the original paper virtually no progress was made until the work of Edelman et al.~\cite{EKS94} who computed the expected number of real eigenvalues. Instead of the determinantal structure known for complex matrices, the eigenvalues of the real Ginibre ensemble form a \textit{Pfaffian} point process~\cite{FN07,SW08,BS09}. Statistics of the real eigenvalues have been shown to arise in many diverse areas, notably in connection to annihilating Brownian motions \cite{TZ11, TKZ12, PTZ17, FTZ20}, random dynamical systems \cite{FK15} and recently to an inverse scattering solution of the Zakharov-Shabat system \cite{BB20}. Real eigenvalues of products have been applied to physical problems \cite{L13,HJL15} and recently appeared in connection to random Markov matrices \cite{IM18}. The analogue of the results discussed above for products of complex matrices also hold here: products of independent real Ginibre matrices are again described by a Pfaffian point process \cite{IK14, FI16}.

Much less is known about truncated orthogonal random matrices and their products. Their study began in the single matrix case with the work of Khoruzhenko et al.~\cite{KSZ10} who showed that the eigenvalues form a Pfaffian point process with an explicit correlation kernel. Unlike the real Ginibre ensemble, there is a qualitatively distinct asymptotic regime known as 
\emph{weak non-orthogonality}, defined by truncating only a finite number of rows and columns from the original orthogonal matrix. This asymptotic regime has shown up in connection to the zeros of random Kac polynomials~\cite{F10}. More recently, it appeared in ~\cite{PS18,GP19} where it is related to the persistence exponent of the 2d diffusion equation with random initial conditions. As with the real Ginibre ensemble, the Pfaffian structure of a single truncated orthogonal random matrix has been extended to products of such matrices \cite{IK14}. Using this integrable structure, the probability that all eigenvalues of a product of truncated orthogonal random matrices are real was calculated in \cite{FK18}. Simplified expressions for the correlation kernel of the Pfaffian point process were obtained in the recent work \cite{FIK20}.

We now discuss the main results of the paper. Let $O$ denote a random orthogonal matrix of size $(N+L) \times (N+L)$ sampled uniformly with respect to the Haar measure on the orthogonal group and let $\tilde{O}$ denote the upper-left $N \times N$ truncation of $O$. We are interested in the product matrix $X^{(m)} := \tilde{O}_{1}\ldots \tilde{O}_{m}$ where $\{\tilde{O}_{j}\}_{j=1}^{m}$ are independent copies of $\tilde{O}$. In our first result we shall make the assumption that $L$ grows in proportion to $N$. Let $L:=L_{N} \to \infty$ be a sequence of positive integers such that 
\begin{equation}
\gamma := \frac{L_{N}}{N} \to  \tilde{\gamma}, \qquad N \to \infty, \label{hyp1}
\end{equation}
where $\tilde{\gamma} > 0$. Throughout the paper we will suppress the $N$-dependence of $L$ and $\gamma$ for notational convenience. We denote by $N_{\mathbb{R}}^{(m)}$ the total number of real eigenvalues of $X^{(m)}$. 

\begin{theorem}
\label{th:mainthm}
Let $L := L_{N}$ be such that hypothesis \eqref{hyp1} holds and define $\alpha := \frac{1}{1+\gamma}$. Then for any fixed $m \in \mathbb{N}$, we have
\begin{equation}
\mathbb{E}(N^{(m)}_{\mathbb{R}}) = \sqrt{\frac{2m\gamma N}{\pi}}\,\mathrm{arctanh}(\sqrt{\alpha}) + O(1), \qquad N \to \infty. \label{mainest}
\end{equation}
\end{theorem}

Next we consider the individual real eigenvalues $\lambda_{1},\ldots,\lambda_{n}$ of $X^{(m)}$, where $n = N^{(m)}_{\mathbb{R}}$. The averaged global density of real eigenvalues is defined by
\begin{equation}
\rho_{N}(x) := \mathbb{E}\left[\sum_{j=1}^{n}\delta(x-\lambda_{j})\right], \label{eigdensdef}
\end{equation}
and we denote its appropriately normalized version $\tilde{\rho}_{N}(x) = \frac{1}{\mathbb{E}(N^{(m)}_{\mathbb{R}})}\rho_{N}(x)$.
\begin{theorem}
\label{th:densthm}
Let $L := L_{N}$ be such that \eqref{hyp1} holds and define $\tilde{\alpha} = \frac{1}{1+\tilde{\gamma}}$. Then for any bounded continuous test function $h$ and fixed $m \in \mathbb{N}$, we have
\begin{equation}
\lim_{N \to \infty}\int_{-1}^{1}h(x)\tilde{\rho}_{N}(x)\,dx = \int_{-1}^{1}h(x)\rho(x)\,dx, \label{hlimit}
\end{equation}
where
\begin{equation}
\rho(x) = \frac{1}{2m\,\mathrm{arctanh}(\sqrt{\tilde{\alpha}})}\,\frac{1}{|x|^{1-\frac{1}{m}}(1-|x|^{\frac{2}{m}})}\mathbbm{1}_{x \in (-\tilde{\alpha}^{\frac{m}{2}},\tilde{\alpha}^{\frac{m}{2}})}. \label{rhodef}
\end{equation}
Furthermore, for any fixed $x \in [-1,1]\setminus \{-\tilde{\alpha}^{\frac{m}{2}},0,\tilde{\alpha}^{\frac{m}{2}}\}$, we have the pointwise limit $\lim_{N \to \infty}\tilde{\rho}_{N}(x) = \rho(x)$.
\end{theorem}
Theorems \ref{th:mainthm} and \ref{th:densthm} prove two conjectures of Forrester, Ipsen and Kumar
\cite[Conjectures 4.1 and 4.2]{FIK20} and generalise the $m=1$ asymptotic results of \cite{KSZ10}. Theorem \ref{th:mainthm} goes further by providing a precise error bound on the remainder, see Remark \ref{rem:remainder} for further discussion on this point. Results of the type \eqref{mainest}, showing that on average there are of order $\sqrt{N}$ real eigenvalues go back to the work of Edelman et al. \cite{EKS94} in the case of a single real Ginibre matrix. Owing to the mentioned developments linking products of real random matrices to Pfaffian point processes in \cite{IK14, FI16}, this was extended to products of independent real Ginibre matrices in \cite{S17}. The `$\sqrt{N}$-law' is expected to hold for a broad class of real random matrices \cite{MPT16}. On this point, for $m=1$ Tao and Vu \cite{TV15} used moment matching methods to extend the $\sqrt{N}$ estimate in \cite{EKS94} to a class of i.i.d. real random matrices. The fluctuations of $N^{(1)}_{\mathbb{R}}$ about the $\sqrt{N}$-law have been investigated and are known to be Gaussian for the real Ginibre ensemble \cite{S17-2, K15}.

The limiting density \eqref{rhodef} has a simple description in terms of its single matrix $m=1$ version. If $A$ is the random variable whose density is given by the $m=1$ case of \eqref{rhodef}, then the symmetrised power $A^{m}B$ gives the density \eqref{rhodef} for any $m \geq 1$, where $B$ is an independent Bernoulli random variable on $\{-1,1\}$. This type of result is familiar in the study of free probability \cite{BNS12} which is very effective at computing the complex spectrum of a product matrix like $X^{(m)}$ in terms of the spectra of the individual factors. However, methods of free probability do not seem to be applicable to the problems solved in this paper regarding real eigenvalues. The main technical challenges in the proofs of Theorems~\ref{th:mainthm} and~\ref{th:densthm} arise in the evaluation 
of integrals \eqref{densiln} and \eqref{iln} in Corollary~\ref{cor:denint}, whose asymptotics are not uniform due to singularities of the integrands near the origin. This problem is common when dealing with products of random matrices and is discussed in more detail in Section~\ref{sec:preest}. It is likely that the technique we develop is useful for other problems involving products of random matrices. For example, our method applies equally well to real Ginibre matrices and this is sketched out in Appendix \ref{sec:realgin}. 

Finally we consider the regime of weak non-orthogonality, defined by fixing $L$ and letting $N \to \infty$. In this case one is only truncating a finite number of rows and columns, and each matrix in the product is very close to an orthogonal matrix. The $\sqrt{N}$-law discussed above does not hold for these matrices. We find
\begin{theorem}
\label{thm:weakregime}
Let $L$ and $m$ be fixed positive integers. Then as $N \to \infty$ we have
\begin{equation}
\mathbb{E}(N^{(m)}_{\mathbb{R}}) \sim \frac{1}{B\left(\frac{mL}{2},\frac{1}{2}\right)}\,\log(N), \label{almostorthogresult}
\end{equation}
where $B(a,b) = \frac{\Gamma(a)\Gamma(b)}{\Gamma(a+b)}$ is the beta function.
\end{theorem}
The case $m=1$ of Theorem \ref{thm:weakregime} was obtained in the work of Khoruzhenko, Sommers and \.{Z}yczkowski \cite{KSZ10}. We are not aware of the general $m$ case given by Theorem \ref{thm:weakregime} appearing in the literature, conjecturally or otherwise.

This paper is structured as follows. In Section \ref{sec:strat} we begin by recalling the results of \cite{FIK20} regarding the Pfaffian structure associated with products of truncated orthogonal random matrices. Then we give an overview of the proof of Theorem \ref{th:mainthm}, postponing the full details until Section \ref{sec:proofs}. Then in Section \ref{sec:convdens} we give the proof of Theorem \ref{th:densthm} based on the results obtained in Section \ref{sec:proofs}. In Section \ref{sec:weakregime} we prove Theorem \ref{thm:weakregime}. Section \ref{sec:laplace} is devoted specifically to the asymptotic analysis of multiple integrals of Laplace type that are needed throughout Sections \ref{sec:strat}, \ref{sec:proofs} and \ref{sec:convdens}. Finally, Appendix \ref{sec:realgin} includes a discussion about how the approach of the present paper is easily adapted to the case of products of real Ginibre matrices.

\section*{Acknowledgements}
A. L. would like to gratefully acknowledge the support of the UK Engineering and Physical Sciences Research
Council (EPSRC) DTP (grant number EP/N509619/1). F. M. is
grateful for support from a University Research Fellowship of the University of Bristol. N. S. gratefully acknowledges support of the Royal Society University Research Fellowship `Random matrix theory and log-correlated Gaussian fields', reference URF\textbackslash R1\textbackslash180707.

\newpage

\section{Strategy of the proof and leading order asymptotics}
\label{sec:strat}
The goal of this section is to discuss the key steps in the proof of Theorem \ref{th:mainthm} and give a quick derivation of the leading asymptotics in \eqref{mainest}. The starting point of the proof is the following fact, first established by Ipsen and Kieburg \cite{IK14}, that the real eigenvalues of the product matrix $X^{(m)} = \tilde{O}_{1} \ldots \tilde{O}_{m}$ form a Pfaffian point process. In the recent work \cite{FIK20} this has been simplified and explicit formulae for the correlation kernel of the Pfaffian point process were identified.
\begin{theorem}[Forrester, Ipsen and Kumar \cite{FIK20}]
\label{thm:FIKtheorem}
Define the function
\begin{equation}
S(x_1,x_2) := \int_{-1}^{1}dy\,(x_{1}-y)\mathrm{sgn}(x_{2}-y)w_{L}(x_{1})w_{L}(y)f_{N-2,L}(x_{1}y), \label{sfunc}
\end{equation}
where the weight function is
\begin{equation}
w_{L}(x) = \left(\frac{L}{2B(\frac{L}{2},\frac{1}{2})}\right)^{\frac{m}{2}}\int_{[-1,1]^{m}}\delta(x-y_{1}\ldots y_{m})\,\prod_{i=1}^{m}(1-y_{i}^{2})^{\frac{L}{2}-1}\,dy_{i}, \label{weight}
\end{equation}
and
\begin{equation}
f_{N-2,L}(x) = \sum_{n=0}^{N-2}\binom{L+n}{n}^{m}x^{n}. \label{fnl}
\end{equation}
Then the real eigenvalues of $X^{(m)}$ form a Pfaffian point process with correlation functions given by 
\begin{equation}
\rho_{(k)}(x_1,\ldots,x_k) = \mathrm{Pf}\left[K(x_i,x_j)\right]_{i,j=1}^{k},
\end{equation}
where the correlation kernel is of the form
\begin{equation}
K(x_i,x_j) = \begin{pmatrix} D(x_{i}, x_{j}) & S(x_{i}, x_{j})\\ -S(x_{j},x_{i}) & \tilde{I}(x_{i},x_{j}) \end{pmatrix}.\label{matrixkernel}
\end{equation}
In \eqref{matrixkernel}, $D$ and $\tilde{I}$ are certain anti-symmetric functions of $x_{i}$ and $x_{j}$ that are not required here; see \cite{FIK20} for further details.
\end{theorem}
In particular, we have
\begin{corollary}
\label{cor:denint}
The one-point density of real eigenvalues $\rho_{N}(x)$ in \eqref{eigdensdef} is obtained by taking $x_{1}=x_{2}=x$ in \eqref{sfunc}, so that $\rho_{N}(x) = S(x,x)$ and therefore
\begin{equation}
\rho_{N}(x) = \int_{-1}^{1}dy\,|x-y|\,w_{L}(x)w_{L}(y)f_{N-2,L}(xy). \label{densiln}
\end{equation}
Furthermore, the expected number of real eigenvalues of $X^{(m)}$ is the total integral of $\rho_{N}(x)$, given by
\begin{equation}
\mathbb{E}(N^{(m)}_{\mathbb{R}}) = 2\int_{0}^{1}dx\int_{-1}^{1}dy\,|x-y|\,w_{L}(x)w_{L}(y)f_{N-2,L}(xy), \label{iln}
\end{equation}
where we used that $\rho_{N}(x)$ is an even function of $x$.
\end{corollary}
\begin{remark}
It is well known that \eqref{weight} has a probabilistic interpretation in terms of products of independent \emph{scalar} random variables. Namely, up to a constant of proportionality, $w_{L}(x)$ is equal to the probability density function for a product of $m$ independent $\mathrm{Beta}\left(\frac{L}{2},\frac{L}{2}\right)$ random variables. Functions of the form \eqref{weight} can also be represented in terms of \emph{Meijer G-functions}, see \cite{ABKN14, FIK20} for further details.
\end{remark}

\subsection{Preliminary estimates}
\label{sec:preest}
To prove Theorems \ref{th:mainthm} and \ref{th:densthm}, the idea is to obtain asymptotics of the functions $w_{L}(x)$ and $f_{N-2,L}(xy)$ separately, and then insert the results into the explicit formulae \eqref{densiln} and \eqref{iln}. Although this would appear to be a straightforward approach, particular care must be taken with the estimates due to parts of the integral \eqref{iln} where the required asymptotics are not uniform. In particular, when $m>1$ both $w_{L}(x)$ and $f_{N-2,L}(xy)$ have singular terms in their asymptotics near the origin, while for $m \geq 1$ the function $f_{N-2,L}(xy)$ undergoes a transition in its asymptotic behaviour near the hyperbola $xy = \alpha^{m}$. 

The purpose of this subsection is simply to detail our asymptotic results for $w_{L}(x)$ and $f_{N-2,L}(xy)$ that are precise enough to prove Theorems \ref{th:mainthm} and \ref{th:densthm}.
\begin{remark}
Throughout the paper, quantities $c$ and $C$ will always denote absolute positive constants that only depend on fixed parameters such as $m$ or $\kappa$ (see below) and do not depend on $N$, $L$ or any variables of integration. Furthermore we deem the precise value of these constants as unimportant and caution that their value may change from line to line.
\end{remark}
\begin{proposition}
\label{prop:uniformweight}
Fix a small constant $\kappa>0$ and a large constant $M>0$. Then the weight function satisfies the following estimate uniformly on $|x| \in [M\,L^{-\frac{m}{2}},1-\kappa]$,
	\begin{equation}
	w_{L}(x) =d_{L,m}(1-|x|^{2/m})^{\frac{mL}{2}-1}|x|^{\frac{1}{m}-1}\left(1+O\left(\frac{1}{|x|^{\frac{2}{m}}L}\right)\right), \qquad L \to \infty, \label{uniformweight}
	\end{equation}
	where $d_{L,m}$ is defined by
\begin{equation}
d_{L,m} = \frac{1}{2}\sqrt{\frac{L}{\pi m}}\,\left(\frac{2\pi}{B(\frac{L}{2},\frac{1}{2})}\right)^{\frac{m}{2}} = \frac{1}{2}\sqrt{\frac{L}{\pi m}}\,(2\pi L)^{\frac{m}{4}}\left(1+O\left(\frac{1}{\sqrt{L}}\right)\right), \qquad L \to \infty. \label{dLm}
\end{equation}
	Furthermore, for $|x| \in [M\,L^{-\frac{m}{2}},1]$ we have the crude estimate
	\begin{equation}
	w_{L}(x) \leq N^{c}(1-|x|^{\frac{2}{m}})^{\frac{mL}{2}}. \label{crudeweight}
	\end{equation}
\end{proposition}

To estimate $f_{N-2,L}(x)$ in \eqref{fnl} the idea will be to replace it with the infinite series
\begin{equation}
f_{\infty,L}(x) := \sum_{n=0}^{\infty}\binom{L+n}{n}^{m}x^{n}. \label{fintLdef}
\end{equation} 
Note that \eqref{fintLdef} converges absolutely inside the unit disc. We will show that the sum \eqref{fnl} naturally splits into two contributions, the first coming from $f_{\infty,L}(x)$ restricted to a specific interval, and the second an error term coming from the tail of the sum.
\begin{proposition}
\label{prop:trunco}
Under the hypotheses of Theorem \ref{th:mainthm} the following holds. We have the following estimate as $N \to \infty$, uniformly on $x \in [-1,1]\setminus((\alpha-\omega)^{m},(\alpha+\omega)^{m})$,
\begin{equation}
f_{N-2,L}(x) = f_{\infty,L}(x)\mathbbm{1}_{-(\alpha+\omega)^{m} < x < (\alpha-\omega)^{m}} + \frac{x^{N-1}}{x-\alpha^{m}}e_{N,m}\left(1+ O\left(\frac{1}{\sqrt{N}\omega}\right)\right), \label{tfintest}
\end{equation}
where $\omega$ may depend on $N$, and
\begin{equation}
e_{N,m} := \frac{\gamma^{-m\gamma N-\frac{m}{2}}(1+\gamma)^{mN(1+\gamma)-\frac{3m}{2}}}{(2\pi N)^{\frac{m}{2}}}. \label{enm}
\end{equation}
In practice we will take $\omega = N^{-\frac{1}{2}}$, such that the big-$O$ term in \eqref{tfintest} is $O(1)$.
\end{proposition}
The function $f_{\infty,L}(x)$ can then be analysed separately in the limit $L \to \infty$.
\begin{proposition}
\label{prop:fint}
Fix a small constant $\kappa>0$ and a large constant $M>0$. Then the function $f_{\infty,L}(x)$ satisfies the following estimate as $L \to \infty$, uniformly on $x \in [ML^{-m},1-\kappa]$,
\begin{equation}
f_{\infty,L}(x) = (1-x^{\frac{1}{m}})^{-mL-1}\,L^{-\frac{m-1}{2}}\,\frac{1}{(2\pi)^{\frac{m-1}{2}}}\,\frac{1}{\sqrt{m}}\,x^{-\frac{m-1}{2m}}\,\left(1+O\left(\frac{1}{x^{\frac{1}{m}}L}\right)\right). \label{fintasytru}
\end{equation}
Furthermore, on the interval $x \in [ML^{-m},1]$ we have the crude estimate
\begin{equation}
f_{\infty,L}(x) \leq C(1-x^{\frac{1}{m}})^{-mL}, \label{crudefint}
\end{equation}
while for $x \in [-1,-ML^{-m}]$ we have
\begin{equation}
|f_{\infty,L}(x)| \leq C(1-(-x)^{\frac{1}{m}})^{-mL}e^{-cL(-x)^{\frac{1}{m}}}. \label{negxbnd}
\end{equation}
\end{proposition}

The proofs of Propositions \ref{prop:uniformweight}, \ref{prop:trunco} and \ref{prop:fint} are given in full in Section \ref{sec:laplace}. As we shall see, all the quantities involved have convenient integral representations and are then amenable to the \textit{Laplace} or \textit{saddle point method} for the asymptotic analysis of multiple integrals. Although this is relatively standard, there is a threshold beyond which the Laplace method breaks down, for example if $0 < x < L^{-\frac{m}{2}-\epsilon}$ for any $\epsilon>0$ in Proposition \ref{prop:uniformweight}. 
\begin{remark}
\label{rem:pointwise}
The decomposition \eqref{tfintest} was inspired by analogous results for the truncated exponential function given in \cite{BG07}. At the level of pointwise asymptotics, the leading terms in \eqref{uniformweight} and \eqref{fintasytru} were also derived in \cite{ABKN14}. However, our situation requires more precise estimates and uniform error bounds as given in the above three Propositions. This is due to the specific integral form in \eqref{sfunc} that is used to construct the correlation kernel for products of truncated orthogonal matrices.
\end{remark}

\subsection{Proof of Theorem \ref{th:mainthm}}
\label{sec:leading}
A straightforward check shows that after formally substituting the asymptotics \eqref{uniformweight} and \eqref{fintasytru} into \eqref{iln}, there is a critical point at $y=x$ in the positive quadrant. Then keeping $xy < (\alpha-N^{-\frac{1}{2}})^{m}$ in accordance with \eqref{tfintest}, our main results are obtained by a final saddle point computation.

To make this approach rigorous we need to remove some regions of the integration in \eqref{iln} where the mentioned asymptotics do not hold or otherwise give negligible contributions. These include a small microscopic layer of width $MN^{-\frac{m}{2}}$ near the origin, a small region near $\pm 1$ of fixed size $\kappa>0$ and the entire negative $y$ quadrant, see Lemmas \ref{lem:micro}, \ref{lem:edge} and \ref{lem:2ndterm} for precise statements. The last of these Lemmas also shows that we can always neglect the second term in \eqref{tfintest}. This results in the following approximation of \eqref{iln}:
\begin{lemma}
\label{lem:prequad}
For any $\kappa>0$ sufficiently small and $M>0$ a fixed large constant, we have $\mathbb{E}(N_{\mathbb{R}}) = I_{+} + O(1)$ as $N \to \infty$, where
\begin{equation}
\begin{split}
I_{+} = 2m^{2}\int_{N^{-1/2}M'}^{1-\kappa'}du\,\int_{N^{-1/2}M'}^{1-\kappa'}dv&\,|v^{m}-u^{m}|w_{L}(u^{m})w_{L}(v^{m})\\
&\times f_{\infty,L}((uv)^{m})\,(uv)^{m-1}\,\mathbbm{1}_{uv < \alpha - N^{-\frac{1}{2}}}, \label{prequadapprox}
\end{split}
\end{equation}
and $M' = M^{\frac{1}{m}}$, $1-\kappa' = (1-\kappa)^{\frac{1}{m}}$.
\end{lemma}
Note that \eqref{prequadapprox} is the integral \eqref{iln} with the domain of integration restricted to a certain subset of the positive quadrant, within which we have made the change of variables $x = u^{m}$ and $y = v^{m}$. Throughout the paper we will frequently work in the $u$ and $v$ coordinates instead of $x$ and $y$ coordinates as this turns out to greatly simplify the presentation. 

The advantage of the approximation \eqref{prequadapprox} is that the asymptotics \eqref{uniformweight} and \eqref{fintasytru} are now applicable and have uniform error bounds. As we explained, we may expect the main contributions to come from a small $\delta$-neighbourhood of the line $v=u$ for some fixed $\delta>0$. This motivates the introduction of the \textit{critical region}:
\begin{equation}
A := \bigg\{(u,v) \in [0,1]^{2} : \{u > N^{-\frac{1}{2}}M'\} \wedge \{v > N^{-\frac{1}{2}}M'\} \wedge \{|u-v| < \delta\} \wedge \{uv < \alpha - N^{-\frac{1}{2}}\}\bigg\}. \label{critreg}
\end{equation}
Combined with Lemma \ref{lem:prequad}, the following Lemma immediately yields the proof of Theorem \ref{th:mainthm}. 
\begin{lemma}
\label{lem:critasympt}
Let $I_{+,A}$ denote the contribution of the integral \eqref{prequadapprox} from the set $A$ in \eqref{critreg}. Then $I_{+} = I_{+,A} + O(e^{-cN})$ and 
\begin{equation}
I_{+,A} = \sqrt{\frac{2m\gamma N}{\pi}}\,\mathrm{arctanh}(\sqrt{\alpha}) + O(1), \qquad N \to \infty. \label{o1error}
\end{equation}
\end{lemma}
For ease of presentation, in this section we will just show that $I_{+,A} \sim \sqrt{\frac{2m\tilde{\gamma} N}{\pi}}\,\mathrm{arctanh}(\sqrt{\tilde{\alpha}})$ where $\tilde{\alpha} := \frac{1}{1+\tilde{\gamma}}$, leaving the full proof of Lemmas \ref{lem:prequad} and \ref{lem:critasympt} to Section \ref{sec:proofs}. Inserting \eqref{uniformweight} and \eqref{fintasytru} into \eqref{prequadapprox} leads to the approximation
\begin{equation}
I_{+,A} = C_{m,L}\int_{N^{-1/2}M'}^{\sqrt{\alpha-N^{-\frac{1}{2}}}}du\,\int_{u}^{\mathrm{min}(u+\delta,\frac{\alpha-N^{-\frac{1}{2}}}{u})}dv\,\frac{v^{m}-u^{m}}{(uv)^{\frac{m-1}{2}}}\,T_{L}(u,v)\left(1+O\left(\frac{1}{u^{2}L}\right)\right),\label{quadapprox}
\end{equation}
where
\begin{equation}
T_{L}(u,v) = \left[\frac{(1-u^{2})(1-v^{2})}{(1-uv)^{2}}\right]^{\frac{m\gamma N}{2}}(1-u^{2})^{-1}(1-v^{2})^{-1}(1-uv)^{-1}. \label{trunckern}
\end{equation}
In \eqref{quadapprox} we used the symmetry of the set $A$ to integrate only over the region $\{v>u\}$. The pre-factor $C_{m,L}$ is the result of collecting all the individual pre-factors from \eqref{fintasytru} and \eqref{uniformweight}, together with a factor $4m^{2}$ from \eqref{prequadapprox} and the symmetry of $A$. Using the standard asymptotics of the beta function, we have
\begin{equation}
C_{m,L} \sim \tilde{\gamma} N\,\sqrt{\frac{2m\tilde{\gamma} N}{\pi}}, \qquad N \to \infty. \label{cmlasy}
\end{equation}
The contribution of the term $O\left(\frac{1}{u^{2}L}\right)$ in \eqref{quadapprox} will be analysed more precisely in the full proof of Lemma \ref{lem:critasympt}. For now observe that since $u > M'N^{-\frac{1}{2}}$ its contribution goes to zero when $N \to \infty$ followed by $M \to \infty$; this is sufficient to neglect it at leading order. 

Proceeding now with the saddle point asymptotics, we see that the function \eqref{trunckern} has a critical point at $v=u$. This motivates the change of coordinates $v \to v/\sqrt{N}+u$ in \eqref{quadapprox} and using \eqref{cmlasy} we obtain
\begin{equation}
\begin{split}
I_{+,A} \sim \gamma\,\sqrt{\frac{2m\gamma N}{\pi}}\int_{N^{-1/2}M'}^{\sqrt{\alpha-N^{-\frac{1}{2}}}}du&\,\int_{0}^{\sqrt{N}\mathrm{min}(\delta,\frac{\alpha-N^{-\frac{1}{2}}}{u}-u)}dv\,\sqrt{N}Q(u,v/\sqrt{N})\\
&\times T_{L}\left(u,u+v/\sqrt{N}\right), \label{shiftiln}
\end{split}
\end{equation}
where
\begin{equation}
Q(u,v) = \frac{(u+v)^{m}-u^{m}}{\left(u\left(u+v\right)\right)^{\frac{m-1}{2}}}. \label{qdefintro}
\end{equation}
Regarding the integrand of \eqref{shiftiln}, we have the easily derived pointwise limits
\begin{align}
\lim_{N \to \infty}T_{L}\left(u,u+v/\sqrt{N}\right) &= \frac{\mathrm{exp}\left(-\frac{m\tilde{\gamma} v^{2}}{2(1-u^{2})^{2}}\right)}{(1-u^{2})^{3}}, \label{tl-limit}\\
\lim_{N \to \infty}\sqrt{N}Q(u,v/\sqrt{N}) &= mv.  \label{qlimit}
\end{align}
By \eqref{gaussbnd} we have the uniform bound $T_{L}(u,u+v/\sqrt{N}) \leq Ce^{-c v^{2}}$. Regarding the function $Q(u,v)$ in \eqref{qdefintro}, we use the Taylor expansion \eqref{qsmallu} and the fact that $u > M'N^{-\frac{1}{2}}$ to see that $\sqrt{N}Q(u,v/\sqrt{N})$ is uniformly bounded by a polynomial in $v$ with finite degree. Hence we can apply the dominated convergence theorem and pass the limit inside the integrals \eqref{shiftiln}, leading to
\begin{align}
I_{+,A} &\sim \tilde{\gamma} \sqrt{\frac{2m\tilde{\gamma} N}{\pi}}\int_{0}^{\sqrt{\tilde{\alpha}}}du\,\int_{0}^{\infty}dv\,mv\frac{\mathrm{exp}\left(-\frac{m\tilde{\gamma} v^{2}}{2(1-u^{2})^{2}}\right)}{(1-u^{2})^{3}}\\
&= \sqrt{\frac{2m\tilde{\gamma} N}{\pi}}\int_{0}^{\sqrt{\tilde{\alpha}}}du\,\frac{1}{1-u^{2}} \label{arctanhderiv}\\
&= \sqrt{\frac{2m\tilde{\gamma} N}{\pi}}\,\mathrm{arctanh}(\sqrt{\tilde{\alpha}}),
\end{align}
which is the desired leading order result for the number of real eigenvalues. 


\section{Proof of Lemmas \ref{lem:prequad} and \ref{lem:critasympt}}
\label{sec:proofs}
We now provide the key estimates that allow us to prove Lemmas \ref{lem:prequad} and \ref{lem:critasympt}. For a subset $E \subset [0,1] \times [-1,1]$, we will denote $I_{E}$ as the integral \eqref{iln} with the domain of integration restricted to $E$. We begin with the proof of Lemma \ref{lem:prequad}. This will be the immediate consequence of the following smaller Lemmas that follow now. 
\begin{lemma}
\label{lem:micro}
Let $M>0$ be a large fixed constant and define
\begin{equation}
E_{1} := \bigg\{(x,y) \in [0,1] \times [-1,1] : \{x < MN^{-\frac{m}{2}}\} \vee \{|y| < MN^{-\frac{m}{2}}\}\bigg\}.
\end{equation}
Then $I_{E_1} = O(1)$ as $N \to \infty$.
\end{lemma}

\begin{proof}
From the definition \eqref{weight}, we have
\begin{equation}
w_{L}(x) = 2^{m-1} \left(\frac{L}{2 B( \frac{L}{2},\frac{1}{2})}\right)^{\frac{m}{2}} \int_{A_x} \left(1 - x^2 \prod_{i=1}^{m-1} y_i^{-2} \right)^{\frac{L}{2}-1} \prod_{i=1}^{m-1} \frac{(1-y_i^2)^{\frac{L}{2}-1}}{y_i} \, d\vec{y},
\end{equation}
where $A_x = \left\{ \vec{y} \in (0,1)^{m-1} \, : \, |x| < \prod_{i=1}^{m-1} |y_i| < 1 \right\}$. Now using the inequality $1-x^{2} \leq e^{-x^{2}}$,
\begin{equation}
w_{L}(x) \leq CL^{\frac{3m}{4}}\int_{[0,\infty)^{m-1}} \mathrm{exp}\left(-\left(\frac{L}{2}-1\right)\left(\frac{x^{2}}{y_{1}^{2}\ldots y_{m-1}^{2}}+\sum_{j=1}^{m-1}y_{j}^{2}\right)\right) \, \frac{d\vec{y}}{y_{1}\ldots y_{m-1}},
\end{equation}
and changing variables $y_{j} \to y_{j}x^{\frac{1}{m}}$ we see that the last integral is precisely the weight function \eqref{ginweightdef} of the real Ginibre ensemble, see the equivalent formula \eqref{ginweightdef2}. We thus have the bound $w_{L}(x) \leq CL^{\frac{3m}{4}}w_{\mathrm{Gin}}((L/2)^{\frac{m}{2}}x)$ where we used that $\frac{L}{2}-1 > \frac{L}{4}$. Furthermore, we bound the sum $|f_{N-2,L}(xy)|$ in terms of the analogous Ginibre sum \eqref{ginfdef} noting that
\begin{equation}
\begin{split}
\sum_{n=0}^{N-2}\binom{L+n}{n}^{m}|xy|^{n} \leq \sum_{n=0}^{N-2}\frac{(L+N)^{nm}|xy|^{n}}{(n!)^{m}} &= \sum_{n=0}^{N-2}\frac{|(1+\gamma)^{m}N^{m}xy|^{n}}{(n!)^{m}},  \label{ginsumbnd}
\end{split}
\end{equation}
so that by definition of \eqref{ginfdef} we have
\begin{equation}
|f_{N-2,L}(xy)| \leq f_{N-2}((1+\gamma)^{m}N^{m}|xy|).
\end{equation}
Then applying the change of variables $x \to xL^{-m/2}$ and $y \to yL^{-m/2}$ exactly cancels the large pre-factors coming from the weights and gives the bound,
\begin{equation}
\begin{split}
I_{E_1} &\leq C\int_{0}^{M'}dx\,\int_{0}^{\infty}dy\,(y+M)w_{\mathrm{Gin}}(x)w_{\mathrm{Gin}}(y)f_{N-2}(c|xy|)\\
&\leq  C\int_{0}^{M}w_{\mathrm{Gin}}(x)\,F(cx)\,dx \label{Ie1bnd}
\end{split}
\end{equation}
where we used $|x-y| < |y|+M$ (the sign of $x$ and $y$ is not relevant here) and defined
\begin{equation}
F(x) := \sum_{j=0}^{\infty}\frac{x^{j}}{(j!)^{m}}\int_{0}^{\infty}dy\,w_{\mathrm{Gin}}(y)y^{j+1}+M\sum_{j=0}^{\infty}\frac{x^{j}}{(j!)^{m}}\int_{0}^{\infty}dy\,w_{\mathrm{Gin}}(y)y^{j}.
\end{equation}
To conclude that $I_{E_1} = O(1)$ it suffices to check that $F(x)$ is bounded on the compact set $[0,M]$. To compute $F(x)$ we use the moment formula
\begin{equation}
\int_{0}^{\infty}dy\,y^{j+1}w_{\mathrm{Gin}}(y) = \left(2^{\frac{j}{2}}\Gamma(j/2+1)\right)^{m}.
\end{equation}
This follows directly from the definition of the Ginibre weight in \eqref{ginweightdef}. Therefore
\begin{equation}
F(x) = \sum_{j=0}^{\infty}\frac{x^{j}}{(j!)^{m}}\,\left(2^{\frac{j}{2}}\Gamma(j/2+1)\right)^{m}+M\sum_{j=0}^{\infty}\frac{x^{j}}{(j!)^{m}}\,\left(2^{\frac{j-1}{2}}\Gamma(j/2+1/2)\right)^{m}.
\end{equation}
It is straightforward to check that the radius of convergence of these power series is infinite. Consequently $F(x)$ defines an entire function of $x$ (in particular it is continuous and bounded on compact sets) and this implies that \eqref{Ie1bnd} is bounded. Therefore $I_{E_1} = O(1)$ as required.
\end{proof}

Now in the other regions outside $E_{1}$ in the positive quadrant, the contribution to \eqref{iln} is exponentially small on any set which is uniformly bounded away from the main diagonal.
\begin{lemma}
\label{lem:crudebnd}
Fix a large constant $M>0$. Uniformly on the domain $(x,y) \in [MN^{-\frac{m}{2}},1]^{2}$, we have the bound
\begin{equation}
|x-y|w_{L}(x)w_{L}(y)f_{N-2,L}(xy) \leq N^{c}e^{-c'N(u-v)^{2}} \label{crudeintegrand}
\end{equation}
where $u = x^{\frac{1}{m}}$, $v = y^{\frac{1}{m}}$. Consequently, if the domain $E \subset [MN^{-\frac{m}{2}},1]^{2}$ is strictly bounded away from the main diagonal $x=y$ by some fixed $\delta>0$ independent of $N$, then we have the exponential decay $I_{E} = O(e^{-cN})$ as $N \to \infty$.
\end{lemma}

\begin{proof}
This follows immediately from putting together the individual bounds \eqref{crudeweight}, $f_{N-2,L}(xy) \leq f_{\infty,L}(xy)$, \eqref{crudefint} and \eqref{gaussbnd}, together with the fact that $|x-y| \leq 1$.
\end{proof}

\begin{lemma}
\label{lem:edge}
Consider the following small neighbourhood of the hyperbola $xy=\alpha^{m}$:
\begin{equation}
E_{2} := \bigg\{(x,y) \in [0,1]^{2} : \{(\alpha-N^{-\frac{1}{2}})^{m} < xy < (\alpha+N^{-\frac{1}{2}})^{m}\}\bigg\}.
\end{equation}
Then we have $I_{E_2} = O(1)$ as $N \to \infty$.
\end{lemma}

\begin{proof}
By symmetry we consider just the part of $E_{2}$ where $y>x$. We will make use of the bound $f_{N-2,L}(xy) \leq f_{\infty,L}(xy)$ throughout the proof. Let us denote $\alpha^{m}_{\pm} := (\alpha \pm N^{-\frac{1}{2}})^{m}$ and split the integration domain as the disjoint union $E_{2}\cap\{y > x\} =\bigcup_{j=1}^{3} E_{2,j}$ where
\begin{align}
E_{2,1} &= \bigg\{\alpha^{m}_{-} < x < t^{m}, \quad \frac{\alpha^{m}_{-}}{x} < y < \frac{\alpha^{m}_{+}}{x}\bigg\},\\
E_{2,2} &= \bigg\{t^{m} < x < \alpha^{\frac{m}{2}}_{-}, \quad \frac{\alpha^{m}_{-}}{x} < y < \frac{\alpha^{m}_{+}}{x}\bigg\},\\
E_{2,3} &= \bigg\{\alpha^{\frac{m}{2}}_{-} < x < \alpha^{\frac{m}{2}}_{+}, \quad x < y < \frac{\alpha^{m}_{+}}{x}\bigg\},
\end{align}
and $t$ is any fixed constant such that $\alpha_{-} < t < \sqrt{\alpha_{-}}$. The main contribution will come from a small neighbourhood of the point $x = \alpha_{-}^{\frac{m}{2}}$, while the contribution from $E_{2,1}$ is exponentially small due to Lemma \ref{lem:crudebnd}. On $E_{2,2}$ and $E_{2,3}$ the asymptotics \eqref{uniformweight} and \eqref{fintasytru} can be applied and we mimic the steps that led to \eqref{quadapprox}. Also shifting $v \to v+u$, applying \eqref{gaussbnd} and Lemma \ref{lem:qbound}, we find
\begin{equation}
\label{E2decomp}
I_{E_{2,2}} \leq CN^{\frac{3}{2}}\int_{t}^{\sqrt{\alpha_{-}}}du\,\int_{\frac{\alpha_{-}}{u}-u}^{\infty}dv\,ve^{-cNv^{2}}, \qquad I_{E_{2,3}} \leq CN^{\frac{3}{2}}\int_{\sqrt{\alpha_{-}}}^{\sqrt{\alpha_{+}}}du\,\int_{0}^{\infty}dv\,ve^{-cNv^{2}}.
\end{equation}
The integrals over $v$ are now explicit and we get
\begin{equation}
I_{E_{2,2}} \leq CN^{\frac{1}{2}}\int_{t}^{\sqrt{\alpha_{-}}}du\,e^{-cN(\frac{\alpha_{-}}{u}-u)^{2}} = O(1), \qquad N \to \infty, \label{E22laplace}
\end{equation}
where the $O(1)$ bound follows from a standard Laplace approximation near the critical point $u = \sqrt{\alpha_{-}}$. Evaluating the second integral in \eqref{E2decomp} we obtain $I_{E_{2,3}} \leq CN^{\frac{1}{2}}(\sqrt{\alpha_{+}}-\sqrt{\alpha_{-}}) = O(1)$ as $N \to \infty$. This completes the proof of the Lemma.
\end{proof}

We now show that the second term in \eqref{tfintest} does not contribute to the leading order asymptotics. Note that for the purposes of proving Lemma \ref{lem:prequad} we do not need to show this on the union $E' = E_{1} \cup E_{2}$ of the two negligible sets of Lemmas \ref{lem:micro} and \ref{lem:edge}.

\begin{lemma}
\label{lem:2ndterm}
Consider the contribution of the second term in \eqref{tfintest} to the integral \eqref{iln} on the set $D := [0,1-\kappa] \times [-1+\kappa,1-\kappa] \setminus E'$, taking $\omega = N^{-\frac{1}{2}}$ in Proposition \ref{prop:trunco}. Then  
\begin{equation}
J := 2\int_{D}dx\,dy\,|x-y|w_{L}(x)w_{L}(y)\,e_{N,m}\frac{(xy)^{N-1}}{xy-\alpha^{m}} = O(1), \qquad N \to \infty. \label{jint}
\end{equation}
Furthermore, consider the following thin layer near $\pm 1$ of width $\kappa>0$,
\begin{equation}
E_{3} := \bigg\{(x,y) \in [0,1] \times [-1,1] : \{x > 1-\kappa\} \vee \{|y| > 1-\kappa\}\bigg\}\setminus E', \label{setE3}
\end{equation}
and finally the remaining part of the negative-$y$ quadrant
\begin{equation}
E_{4} := ([0,1] \times [-1,0]) \setminus (E'\cup E_{3}). \label{setE4}
\end{equation}
Then for $\kappa>0$ sufficiently small we have $I_{E_3} = O(e^{-cN})$ and $I_{E_4} = O(1)$ as $N \to \infty$.
\end{lemma}

\begin{proof}
We begin with the estimate \eqref{jint}. Write $J = J_{+} + J_{-}$ where $J_{+}$ (resp. $J_{-}$) denotes the contribution to $J$ from $y>0$ (resp. $y<0$). First consider $J_{+}$, on which we bound the term $|(xy-\alpha^{m})^{-1}| \leq C\sqrt{N}$ due to the fact that the interval $(\alpha-N^{-\frac{1}{2}})^{m} < xy < (\alpha+N^{-\frac{1}{2}})^{m}$ lies outside the integration domain $D$. Inserting the asymptotics of the weights \eqref{uniformweight} and extending the domain of integration to $[0,1]^{2}$, we arrive at
\begin{equation}
\begin{split}
&J_{+} \leq CN^{\frac{m+3}{2}}e_{N,m}\int_{0}^{1}du\,\int_{0}^{1}dv\,|u^{m}-v^{m}|(1-u^{2})^{\frac{m\gamma N}{2}-1}(1-v^{2})^{\frac{m\gamma N}{2}-1}(uv)^{m(N-1)}\\
&\sim CN^{\frac{m+3}{2}}e_{N,m}\int_{\sqrt{\alpha}-\epsilon}^{\sqrt{\alpha}+\epsilon}du\,\int_{\sqrt{\alpha}-\epsilon}^{\sqrt{\alpha}+\epsilon}dv\,|u^{m}-v^{m}|(1-u^{2})^{\frac{m\gamma N}{2}-1}(1-v^{2})^{\frac{m\gamma N}{2}-1}(uv)^{m(N-1)}. \label{localize2ndterm}
\end{split}
\end{equation}
To arrive at the second line of \eqref{localize2ndterm} we have noted that $\sqrt{\alpha}$ is a critical point of the integrand in both $u$ and $v$ variables, so that contributions from outside the small $\epsilon$-neighbourhood of $\sqrt{\alpha}$ are exponentially suppressed. Applying the standard techniques of the Laplace method, we obtain two factors of $N^{-\frac{1}{2}}$ coming from Taylor expanding near the critical point, and an extra factor of $N^{-\frac{1}{2}}$ from the term $|u^{m}-v^{m}|$. Then the leading contribution from evaluating the integrand of \eqref{localize2ndterm} at both critical points gives an upper bound of order
\begin{equation}
O\left(N^{\frac{m}{2}}e_{N,m}(1-\alpha)^{m\gamma N}\alpha^{mN}\right) = O(1),
\end{equation}
where to deduce the $O(1)$ bound, we inserted the explicit form of $e_{N,m}$ defined in \eqref{enm}, keeping in mind that $\alpha := \frac{1}{1+\gamma}$. 

The procedure for $J_{-}$ is exactly the same as for $J_{+}$ except that since now $xy$ is negative, we can bound $|xy-\alpha^{m}|^{-1}$ by an absolute constant. For the same reason the term $|x-y|$ in \eqref{iln} is of no use in the bounds and we use the simple fact that $|x-y| \leq 2$. This results in an upper bound 
\begin{equation}
J_{-} \leq CN^{\frac{m}{2}+1}e_{N,m}\int_{\sqrt{\alpha}-\epsilon}^{\sqrt{\alpha}+\epsilon}du\,\int_{\sqrt{\alpha}-\epsilon}^{\sqrt{\alpha}+\epsilon}dv\,(1-u^{2})^{\frac{m\gamma N}{2}-1}(1-v^{2})^{\frac{m\gamma N}{2}-1}(uv)^{m(N-1)} = O(1), \label{jmin}
\end{equation}
as $N \to \infty$, by the same reasoning. This completes the proof of \eqref{jint}.

Regarding the set $E_{3}$ in \eqref{setE3}, in the above bounds for $J_{+}$ and $J_{-}$ we saw that the second term in \eqref{tfintest} can be neglected outside $E'$ and gives an $O(1)$ contribution. In fact repeating these estimates restricted to the set $E_{3}$ gives a contribution of order $O(e^{-cN})$ because $E_{3}$ lies outside the saddle point region discussed below \eqref{localize2ndterm}. For this to be true we must choose $\kappa>0$ small enough such that $1-\kappa > \sqrt{\alpha}$ for all $N$ sufficiently large, which is guaranteed by hypothesis \eqref{hyp1}. Thus it remains to show the same estimate for the first term in \eqref{tfintest}. For $y>0$, the indicator function in \eqref{tfintest} combined with Lemma \ref{lem:crudebnd} gives the required exponential decay because $E_{3} \cap \{xy < (\alpha-N^{-\frac{1}{2}})^{m}\}$ is uniformly bounded away from the main diagonal. When $y<0$ we similarly note that $E_{3} \cap \{xy > -(\alpha+N^{-\frac{1}{2}})^{m}\}$ is uniformly bounded away from the anti-diagonal $x=-y$. Then $I_{E_3 \cap \{y<0\}} = O(e^{-cN})$ follows as in Lemma \ref{lem:crudebnd}, using \eqref{negxbnd} in place of \eqref{crudefint}.

Finally we show that $E_{4}$ in \eqref{setE4} is negligible. By \eqref{jint} we can ignore the second term in \eqref{tfintest} and deal only with the first term $f_{\infty,L}(xy)$. When $y<0$ we change variables with $x = u^{m}$ and $y=-(v^{m})$. Making use of the bound \eqref{negxbnd}, we insert the asymptotics of the weights \eqref{uniformweight}. The combination of the main terms in \eqref{uniformweight} and \eqref{negxbnd} is bounded using \eqref{gaussbnd}, namely for any $\delta>0$ we have
\begin{equation}
\begin{split}
\left(\frac{(1-u^{2})(1-v^{2})}{(1-uv)^{2}}\right)^{\frac{mL}{2}}e^{-cLuv} &\leq e^{-\frac{mL}{2}(u-v)^{2}-c Luv}\\
&= e^{-\delta L(u^{2}+v^{2}) - L(\frac{m}{2}-\delta)(u-v)^{2}-L(c-2\delta)uv}\\
& \leq e^{-\delta L(u^{2}+v^{2})}, \label{negvest}
\end{split}
\end{equation}
where we took $0 < \delta < \mathrm{min}(\frac{m}{2},\frac{c}{2})$ in the final inequality. This gives
\begin{equation}
I_{E_{4}} \leq CL^{\frac{m}{2}+1}\int_{0}^{1}du\,\int_{0}^{1}dv\,(u^{m}+v^{m})e^{-\delta L(u^{2}+v^{2})}+J = O(1),
\end{equation}
and this completes the proof of Lemma \ref{lem:2ndterm}.
\end{proof}

\begin{proof}[Proof of Lemma \ref{lem:prequad}]
This follows immediately from combining the three Lemmas \ref{lem:micro}, \ref{lem:edge}, \ref{lem:2ndterm} and noting the presence of the indicator function $xy < (\alpha-N^{-\frac{1}{2}})^{m}$ in Proposition \ref{prop:trunco}.
\end{proof}
\begin{proof}[Proof of Lemma \ref{lem:critasympt}]
We start from the expression \eqref{quadapprox}, which after the shift $v \to v+u$ we can write as
\begin{equation}
I_{+,A} = C_{m,L}\,\int_{N^{-1/2}M'}^{\sqrt{\alpha-N^{-\frac{1}{2}}}}du\,\int_{0}^{\mathrm{min}(\delta,\frac{\alpha-N^{-\frac{1}{2}}}{u}-u)}dv\,F(u,v)e^{\frac{m\gamma N}{2}\phi(u,v)}\left(1+O\left(\frac{1}{u^{2}L}\right)\right), \label{quadapproxpf}
\end{equation}
where
\begin{equation}
\phi(u,v) = \log\left(\frac{(1-u^{2})(1-(v+u)^{2})}{(1-u^{2}(v+u)^{2})}\right),
\end{equation}
and in terms of \eqref{qdefintro},
\begin{equation}
F(u,v) = Q(u,v)\,(1-u^{2})^{-1}(1-(v+u)^{2})^{-1}(1-u(v+u))^{-1}.
\end{equation}
The pre-factor outside \eqref{quadapprox} satisfies 
\begin{equation}
C_{m,L} = \gamma N\,\sqrt{\frac{2m\gamma N}{\pi}}\left(1+O\left(\frac{1}{\sqrt{N}}\right)\right), \qquad N \to \infty, \label{cml2}
\end{equation}
where we made use of \eqref{dLm}. Let us denote $I^{(1)}_{+,A}$ the contribution of \eqref{quadapproxpf} where we discard the error terms $O(N^{-\frac{1}{2}})$ in \eqref{cml2} and $O\left(\frac{1}{u^{2}L}\right)$ in \eqref{quadapproxpf}; we will explain at the end of the proof why $I_{+,A} = I^{(1)}_{+,A} + O(1)$ as $N \to \infty$. We begin by Taylor expanding $F(u,v)$ and the action $\phi(u,v)$ near $v=0$. A careful estimation of the remainder, see the bound \eqref{qsmallu}, shows that uniformly on $u \in [0,1-\kappa)$, we have
\begin{equation}
F(u,v) = \frac{mv}{(1-u^{2})^{3}} + O(v^{2}) + \sum_{j=0}^{m-1}O\left(\frac{v^{3+j}}{u^{2+j}}\right), \qquad 0 < v < \delta. \label{Fxvexpan}
\end{equation}
Note that the singular powers of $u$ in \eqref{Fxvexpan} should be handled carefully. We will show below that they contribute at most $O(1)$ as $N \to \infty$, due to the fact that the $u$-integration starts at $N^{-\frac{1}{2}}M'$, meanwhile the high powers of $v$ give successively smaller contributions to the Laplace asymptotics. In particular it is important that the summation in \eqref{Fxvexpan} starts at $v^{3}/u^{2}$ and not $v^{2}/u$, as the latter would give an estimate of order $O(\log(N))$ instead of $O(1)$.

For the action we have $\phi(u,v) = -\frac{v^{2}}{(1-u^{2})^{2}}+E_{3}(u,v)$ where $E_{3}(u,v)$ is a remainder satisfying $E_{3}(u,v) = O(v^{3})$ uniformly on $u \in [0,1-\kappa)$. Let us denote $\xi = \frac{m\gamma N}{2}E_{3}(u,v)$ and $F_{0} = \frac{mv}{(1-u^{2})^{3}}$ in what follows. Then we make the decomposition
\begin{equation}
e^{\xi}F = F_{0} + (F-F_{0})e^{\xi} + (e^{\xi}-1)F_{0} \label{xidecompF}
\end{equation}
and thus write $I^{(1)}_{+,A}$ as the sum of three separate contributions coming from each term in \eqref{xidecompF}: $I^{(1)}_{+,A} = J_{1}+J_{2}+J_{3}$.

For the leading term $J_{1}$, the integral over $v$ is explicit and we obtain
\begin{equation}
\begin{split}
J_{1} = \sqrt{\frac{2 m \gamma N}{\pi}}\,\int_{N^{-1/2}M'}^{\sqrt{\alpha-N^{-\frac{1}{2}}}}du\,\frac{1}{1-u^{2}}-\sqrt{\frac{2 m \gamma N}{\pi}}\,\int_{N^{-1/2}M'}^{\sqrt{\alpha-N^{-\frac{1}{2}}}}du\,\frac{e^{-\frac{m\gamma N}{2}(\mathrm{min}(\delta,\frac{\alpha-N^{-\frac{1}{2}}}{u}-u))^{2}}}{1-u^{2}}. \label{inti2expl}
\end{split}
\end{equation}
In the second integral of \eqref{inti2expl}, define by $u^{*}_{-}$ the unique positive solution of the equation
\begin{equation}
\frac{\alpha-N^{-\frac{1}{2}}}{u} = u+\delta.
\end{equation}
For $u < u^{*}_{-}$ the minimum in \eqref{inti2expl} is $\delta$ and this yields an exponentially small contribution. For $u > u^{*}_{-}$, the minimum is $\frac{\alpha-N^{-\frac{1}{2}}}{u}-u$, which gives an integral of the form \eqref{E22laplace}, so that the second term in \eqref{inti2expl} is $O(1)$. Meanwhile the first integral in \eqref{inti2expl} clearly gives the correct leading term up to errors of order $O(1)$ and so
\begin{equation}
J_{1} = \sqrt{\frac{2 m \gamma N}{\pi}}\,\mathrm{arctanh}(\sqrt{\alpha}) + O(1), \qquad N \to \infty.
\end{equation}
For $J_{2}$ we use \eqref{gaussbnd} and \eqref{Fxvexpan}. The contribution from the error terms in \eqref{Fxvexpan} can then be upper bounded by integrals of the form
\begin{equation}
\sum_{j=0}^{m-1}N^{\frac{3}{2}}\int_{N^{-\frac{1}{2}}M'}^{1}du\,\frac{1}{u^{2+j}}\int_{0}^{\infty}dv\,v^{3+j}e^{-cNv^{2}}\,dv = O(1), \qquad N \to \infty \label{sumo1}
\end{equation}
and
\begin{equation}
N^{\frac{3}{2}}\int_{N^{-\frac{1}{2}}M'}^{1}du\,\int_{0}^{\infty}dv\,v^{2}e^{-cNv^{2}}\,dv = O(1), \qquad N \to \infty.
\end{equation}
Therefore $J_{2} = O(1)$. For $J_{3}$ we use $|e^{\xi}-1| \leq |\xi| e^{|\xi|} \leq CNv^{3}e^{cN\delta v^{2}}$. Hence, choosing $\delta>0$ small enough, the contribution of this error term gives rise to an integral of the form
\begin{equation}
N^{\frac{5}{2}}\int_{N^{-\frac{1}{2}}M'}^{1}du\,\int_{0}^{\infty}dv\,v^{4}e^{-cNv^{2}}\,dv = O(1), \qquad N \to \infty,
\end{equation}
and hence $J_{3} = O(1)$. Finally, if we had included the error term $O\left(\frac{1}{Lu^{2}}\right)$, we would gain an additional power of $N^{-\frac{1}{2}}$ in all the estimates, due to the fact that $\int_{N^{-\frac{1}{2}}M'}^{1}\frac{1}{Lu^{2}}\,du = O(N^{-\frac{1}{2}})$ and similarly in \eqref{sumo1}. The same is true had we included the error term of order $N^{-\frac{1}{2}}$ in \eqref{cml2}. Therefore these error terms can only contribute at most $O(1)$ to $I_{+,A}$ and this completes the proof of \eqref{o1error}. The fact that $I_{+} = I_{+,A} + O(e^{-cN})$ is an immediate consequence of Lemma \ref{lem:crudebnd}. This completes the proof of Lemma \ref{lem:critasympt}.
\end{proof}

\section{Convergence of the eigenvalue density}
\label{sec:convdens}
The goal of this section is to prove Theorem \ref{th:densthm}. Recall that
\begin{equation}
\rho_{N}(x) = w_{L}(x)\int_{-1}^{1}dy\,|x-y|w_{L}(y)f_{N-2,L}(xy). \label{rhodens}
\end{equation}
\begin{proof}[Proof of \eqref{hlimit}]
Since $h$ is bounded and the integrand of \eqref{densiln} is positive, by Lemmas \ref{lem:prequad} and \ref{lem:critasympt} we immediately find that 
\begin{equation}
\int_{-1}^{1}h(x)\tilde{\rho}_{N}(x)\,dx = \frac{I_{-,A}(h)+I_{+,A}(h)}{\mathbb{E}(N^{(m)}_{\mathbb{R}})}+O(N^{-\frac{1}{2}}), \qquad N \to \infty,
\end{equation}
where the analogue of \eqref{prequadapprox} is
\begin{equation}
\begin{split}
I_{\pm,A}(h) &= \frac{1}{2}\,C_{m,L}\int_{N^{-\frac{1}{2}}M'}^{\sqrt{\alpha-N^{-\frac{1}{2}}}}du\,h(\pm u^{m})\,\int_{u}^{\mathrm{min}(u+\delta,\frac{\alpha-N^{-\frac{1}{2}}}{u})}dv\,\frac{v^{m}-u^{m}}{(uv)^{\frac{m-1}{2}}}\,T_{L}(u,v)\left(1+O\left(\frac{1}{u^{2}L}\right)\right)\\
&+\frac{1}{2}\,C_{m,L}\int_{N^{-\frac{1}{2}}M'}^{\sqrt{\alpha-N^{-\frac{1}{2}}}}du\,h(\pm u^{m})\,\int_{\mathrm{max}(N^{-\frac{1}{2}}M',u-\delta)}^{u}dv\,\frac{u^{m}-v^{m}}{(uv)^{\frac{m-1}{2}}}\,T_{L}(u,v)\left(1+O\left(\frac{1}{v^{2}L}\right)\right).
\end{split}
\end{equation}
Finally the limit \eqref{hlimit} follows from changing coordinates $v \to u+\frac{v}{\sqrt{N}}$ (similarly to \eqref{shiftiln}) and again using dominated convergence with the pointwise limits \eqref{tl-limit} and \eqref{qlimit}.
\end{proof}
To obtain the pointwise limit of $\rho_{N}(x)$ requires some further estimates, but they follow a similar pattern to the proofs of Section \ref{sec:proofs}. As in Lemma \ref{lem:micro}, we start by showing that the small region near $y=0$ can be removed from the integration in \eqref{rhodens}.

\begin{lemma}
For any fixed $x \in (0,1]$ we have
\begin{equation}
w_{L}(x)\int_{-MN^{-\frac{m}{2}}}^{MN^{-\frac{m}{2}}}dy\,|x-y|w_{L}(y)f_{N-2,L}(xy) = O(e^{-cN}), \qquad N \to \infty. \label{smallvexp}
\end{equation}
\end{lemma}

\begin{proof}
We start with the case that $M_{2}N^{-m} < y < MN^{-\frac{m}{2}}$ for a large fixed $M_{2}>0$. We bound $|x-y| \leq 2$ and $|f_{N-2,L}(xy)| \leq f_{\infty,L}(|xy|)$. In this sense the sign of $y$ is irrelevant and we assume $y>0$. Then we mimic the proof of Lemma \ref{lem:crudebnd} noting that
\begin{equation}
\bigg{|}w_{L}(x)\int_{M_{2}N^{-m}}^{MN^{-\frac{m}{2}}}dy\,w_{L}(y)f_{N-2,L}(xy) \bigg{|}\leq N^{c}\int_{0}^{MN^{-\frac{1}{2}}}dv\,w_{L}(v^{m})e^{-N(u-v)^{2}}
\end{equation}
where $x = u^{m}$. Since $x>0$ is fixed, so is $u>0$, and we have that $u-v$ is uniformly bounded away from zero on this range of $v$ and \eqref{smallvexp} follows. When $0 < y < M_{2}N^{-m}$ we use \eqref{ginsumbnd} so that
\begin{equation}
f_{N-2,L}(|xy|) \leq f_{N-2}(cL^{m}|xy|) \leq f_{\infty}(c|xy|L^{m}) < C,
\end{equation}
where the last bound follows from the fact that for $x>0$ fixed, $|xy|L^{m}$ is bounded and $f_{\infty}$ has infinite radius of convergence. Now \eqref{smallvexp} follows from \eqref{crudeweight}.
\end{proof}

\begin{lemma}
For any fixed $x>0$ with $x \neq \sqrt{\tilde{\alpha}}$, we have the following estimates as $N \to \infty$,
\begin{align}
&w_{L}(x)\int_{|y|>MN^{-\frac{m}{2}}}dy\,|x-y|w_{L}(y)\,e_{N,m}\frac{(xy)^{N-1}}{xy-\alpha^{m}} = O(e^{-cN}), \label{jdens2}\\
&w_{L}(x)\int_{y < -MN^{-\frac{m}{2}}}dy\,|x-y|w_{L}(y)f_{\infty,L}(xy) = O(e^{-cN}), \label{iminusx}
\end{align}
Furthermore, for any $x > (\tilde{\alpha})^{\frac{m}{2}}$, we have $\rho_{N}(x) = O(e^{-cN})$ and 
\begin{equation}
\rho_{N}(x) = \frac{m}{x^{1-\frac{1}{m}}}\,\int_{u-\epsilon}^{u+\epsilon}dv\,|u^{m}-v^{m}|w_{L}(v^{m})w_{L}(u^{m})f_{\infty}((uv)^{m})\,(uv)^{m-1} + O(e^{-cN}). \label{deltanbhd}
\end{equation}
\end{lemma}

\begin{proof}
Since the order of magnitude on the right-hand side of these bounds is at the exponentially small scale, it suffices to make use of the crude estimates \eqref{crudeweight}, \eqref{crudefint} and \eqref{gaussbnd} (or \eqref{negxbnd} if $y<0$). For example, the estimate \eqref{jdens2} follows by repeating the first steps in the proof of Lemma \ref{lem:2ndterm}, noting that since $x$ is strictly away from the critical point located at $x=(\alpha)^{\frac{m}{2}}$ we only obtain exponentially small contributions, as explained below \eqref{localize2ndterm}. For \eqref{iminusx} we repeat the estimates in \eqref{negvest}. Because the integral over $u$ is absent and $u = x^{\frac{1}{m}} > 0$ is fixed we again obtain exponentially small contributions for \eqref{iminusx}. Next, when $x > (\tilde{\alpha})^{\frac{m}{2}}$, the indicator function in \eqref{tfintest} shows that there exists $\delta>0$ independent of $N$ such that $y <  \tilde{\alpha}^{\frac{m}{2}}-\delta$. Hence $|y-x|$ is uniformly bounded away from zero and $\rho_{N}(x) = O(e^{-cN})$ follows as in Lemma \ref{lem:crudebnd}. The estimate \eqref{deltanbhd} follows similarly.
\end{proof}

\begin{corollary}
If $0 < x < (\tilde{\alpha})^{\frac{m}{2}}$, we have
\begin{equation}
\lim_{N \to \infty}\frac{\rho_{N}(x)}{\mathbb{E}(N^{(m)}_{\mathbb{R}})} = \frac{1}{2m\,\mathrm{arctanh}(\sqrt{\tilde{\alpha}})}\,\frac{1}{x^{1-\frac{1}{m}}(1-x^{\frac{2}{m}})}. \label{limitdens}
\end{equation}
\end{corollary}

\begin{proof}
This follows from the estimate \eqref{deltanbhd} and repeating the steps in Section \ref{sec:leading}, starting at \eqref{prequadapprox} and leading to \eqref{arctanhderiv}.
\end{proof}

\section{Regime of weak non-orthogonality}
\label{sec:weakregime}
In this section we will consider the opposite case of \eqref{hyp1}, that is where $L$ is a \textit{fixed} positive integer, while the size of the original orthogonal matrix from which we truncate grows to infinity with $N$. In this setting the matrices in the product are very close to being orthogonal, having had just a few rows or columns removed. 

We will employ a slightly different approach based on earlier work of the third author \cite{S17} and on an alternative exact formula for the expected number of real eigenvalues in \cite{FIK20}. Specifically, it is shown there that the double integral in \eqref{iln} can be computed as,
\begin{equation}
\mathbb{E}(N^{(m)}_{\mathbb{R}}) = 2q_{L,m}\sum_{j=0}^{N-2}\binom{L+j}{L}^{m}(-1)^{j}g_{j}, \qquad q_{L,m} = \left(\frac{L\Gamma(\frac{L}{2})\Gamma(\frac{L+1}{2})}{2\sqrt{\pi}}\right)^{m}, \label{fikform}
\end{equation}
where $g_{j}$ is a particular instance of a Meijer G-function,
\begin{equation}
g_{j} = \MeijerG*{m+1}{m}{2m+1}{2m+1}{\frac{1}{2},\ldots,\frac{1}{2};\frac{L}{2}+j+1,\ldots,\frac{L}{2}+j+1,\lceil \frac{j}{2} \rceil+1}{\lceil \frac{j}{2} \rceil,j+1,\ldots,j+1;\frac{1-L}{2},\ldots,\frac{1-L}{2}}{1}.
\end{equation}
We refer to \cite{FIK20} and references therein for the precise definition of the Meijer G-function and its basic properties. What is needed here is the following contour integral representation which follows directly from the definition,
\begin{equation}
g_{j} = \frac{1}{2\pi i}\int_{\gamma}\,\left(\frac{\Gamma(\frac{1}{2}+s)\Gamma(j+1-s)}{\Gamma(\frac{L+1}{2}+s)\Gamma(\frac{L}{2}+1+j-s)}\right)^{m}\,\frac{1}{\lceil \frac{j}{2} \rceil-s}\,ds, \label{gjcont}
\end{equation}
where we may take the contour as the imaginary axis, $\gamma = i\mathbb{R}$. The Meijer G-function coefficients can alternatively be written as a real integral as follows.
\begin{lemma}
\begin{equation}
g_{j} = \frac{1}{\Gamma\left(\frac{L}{2}\right)^{2m}}\,\int_{[0,1]^{2m}}\prod_{\ell=1}^{m}(1-t_{\ell})^{\frac{L}{2}-1}(1-r_{\ell})^{\frac{L}{2}-1}t_{\ell}^{\lceil \frac{j}{2} \rceil-\frac{1}{2}}r_{\ell}^{j-\lceil \frac{j}{2} \rceil}\,\mathbbm{1}_{r_{1} \ldots r_{m} > t_{1} \ldots t_{m}}\,d\vec{t}\,d\vec{r}. \label{gjrep}
\end{equation}
\end{lemma}

\begin{proof}
We have the integral representations
\begin{equation}
\begin{split}
\left(\frac{\Gamma\left(\frac{1}{2}+s\right)}{\Gamma(\frac{1+L}{2}+s)}\right)^{m} &= \frac{1}{\Gamma\left(\frac{L}{2}\right)^{m}}\,\int_{[0,1]^{m}}d\vec{t}\,\prod_{\ell=1}^{m}t_{\ell}^{s-\frac{1}{2}}(1-t_{\ell})^{\frac{L}{2}-1},\\
\left(\frac{\Gamma(j+1-s)}{\Gamma(\frac{L}{2}+1+j-s)}\right)^{m} &= \frac{1}{\Gamma\left(\frac{L}{2}\right)^{m}}\,\int_{[0,1]^{m}}d\vec{r}\,\prod_{\ell=1}^{m}r_{\ell}^{j-s}(1-r_{\ell})^{\frac{L}{2}-1}. \label{betaintreps}
\end{split}
\end{equation}
Then inserting \eqref{betaintreps} into \eqref{gjcont} and interchanging the order of integration, the proof is complete after noting the inverse Mellin transform
\begin{equation}
\frac{1}{2\pi i}\int_{\gamma}\left(\prod_{\ell=1}^{m}\frac{t_{\ell}}{r_{\ell}}\right)^{s}\,\frac{ds}{\lceil \frac{j}{2} \rceil-s} = \left(\prod_{\ell=1}^{m}\frac{t_{\ell}}{r_{\ell}}\right)^{\lceil \frac{j}{2} \rceil}\,\mathbbm{1}_{r_{1} \ldots r_{m} > t_{1} \ldots t_{m}}.
\end{equation}
\end{proof}

\begin{lemma}
\label{lem:evenoddest}
We have the following estimates for even and odd indices:
\begin{equation}
\label{gjexpan}
\begin{split}
j^{mL}g_{2j} &= j^{mL}g_{j}^{\mathrm{sym}} + \frac{A_{L,m}}{j} + o\left(\frac{1}{j}\right), \qquad j \to \infty\\
j^{mL}g_{2j+1} &= j^{mL}g_{j}^{\mathrm{sym}}  - \frac{A_{L,m}}{j} + o\left(\frac{1}{j}\right), \qquad j \to \infty.
\end{split}
\end{equation}
where
\begin{equation}
g_{j}^{\mathrm{sym}} = \frac{1}{2}\left(\frac{\Gamma(j+1)}{\Gamma(j+1+\frac{L}{2})}\right)^{2m} \label{gjsym}
\end{equation}
and
\begin{equation}
A_{L,m} = \frac{1}{\Gamma\left(\frac{L}{2}\right)^{2m}}\int_{\mathbb{R}_{+}^{2m}}d\vec{t}\,d\vec{r}\,\prod_{\ell=1}^{m}t_{\ell}^{\frac{L}{2}-1}r_{\ell}^{\frac{L}{2}-1}e^{-t_{\ell}-r_{\ell}}\,\frac{t_{1}+\ldots+t_{m}}{2}\,\mathbbm{1}_{r_{1}+\ldots+r_{m} < t_{1}+\ldots+t_{m}}. \label{alm}
\end{equation}
\end{lemma}

\begin{proof}
We first prove the expansion of $g_{2j}$ in \eqref{gjexpan}. In the integral representation \eqref{gjrep}, we replace the term $\prod_{\ell=1}^{m}t_{\ell}^{-\frac{1}{2}}$ with $1+(\prod_{\ell=1}^{m}t_{\ell}^{-\frac{1}{2}}-1)$, considering the second term as a perturbation. The contribution of the main term in \eqref{gjrep} leads to an integral with permutation invariance between the $r_{\ell}$ and $t_{\ell}$ variables. Hence the indicator function can be dropped after multiplying by $\frac{1}{2}$. The integrals are then explicit and this gives \eqref{gjsym}. In the perturbation term, we change coordinates $t_{\ell} \to 1-\frac{t_{\ell}}{j}$ and $r_{\ell} \to 1-\frac{r_{\ell}}{j}$ and apply the dominated convergence theorem to take the limit $j \to \infty$ inside the integrals, using that
\begin{equation}
\lim_{j \to \infty}j\left(\prod_{\ell=1}^{m}\left(1-\frac{t_{\ell}}{j}\right)^{-\frac{1}{2}}-1\right) = \frac{t_{1}+\ldots+t_{m}}{2}. \label{pertlimit}
\end{equation}
When $j$ is odd the procedure is exactly the same, except the term that breaks the symmetry is $1+(\prod_{\ell=1}^{m}t_{\ell}^{\frac{1}{2}}-1)$. Then we take the same limit in \eqref{pertlimit} except the exponent is $+\frac{1}{2}$ and this results in an overall minus sign on the right-hand side of \eqref{pertlimit}. This completes the proof of the expansions \eqref{gjexpan}.
\end{proof}

\begin{lemma}
For any $L, m$ fixed, we have as $N \to \infty$,
\begin{equation}
\mathbb{E}(N^{(m)}_{\mathbb{R}}) \sim \log(N)\,2q_{L,m}\,\frac{2^{mL}}{(L!)^{m}}\,\left(2A_{L,m}-\frac{mL}{4}\right). \label{logNlemma}
\end{equation}
\end{lemma}

\begin{proof}
We decompose \eqref{fikform} as a sum over even and odd indices via $\mathbb{E}(N_{\mathbb{R}}) = T_{\mathrm{even}}-T_{\mathrm{odd}}$ and use Lemma \ref{lem:evenoddest}. We have
\begin{align}
T_{\mathrm{even}} &= 2q_{L,m}\sum_{j=0}^{\lfloor \frac{N-2}{2} \rfloor}\binom{L+2j}{L}^{m}g_{2j}\\
&=2q_{L,m}\sum_{j=j_{0}}^{\lfloor \frac{N-2}{2} \rfloor}\frac{2^{mL}}{(L!)^{m}}\left(1+\frac{mL(L+1)}{4j}\right)\left(\frac{1}{2}+\frac{\beta}{j}+\frac{A_{L,m}}{j}+o\left(\frac{1}{j}\right)\right) \label{teven}
\end{align}
for some constant $\beta$ whose precise value is unimportant as we shall see below. To derive this we have used the asymptotic expansion of a binomial coefficient and a similar expansion to compute asymptotics of the coefficient $g^{(\mathrm{sym})}_{j}$ in \eqref{gjsym}. Similarly, we have
\begin{equation}
T_{\mathrm{odd}} = 2q_{L,m}\sum_{j=j_{0}}^{\lfloor \frac{N-2}{2} \rfloor}\frac{2^{mL}}{(L!)^{m}}\left(1+\frac{mL(L+3)}{4j}\right)\left(\frac{1}{2}+\frac{\beta}{j}-\frac{A_{L,m}}{j}+o\left(\frac{1}{j}\right)\right). \label{todd}
\end{equation}
Then directly computing the difference $T_{\mathrm{even}}-T_{\mathrm{odd}}$ from \eqref{teven} and \eqref{todd} results in several cancellations so that only terms proportional to $\frac{1}{j}$ remain. Collecting the terms results in \eqref{logNlemma}.
\end{proof}	

\begin{proof}[Proof of Theorem \ref{thm:weakregime}]
In the Appendix we compute the integrals $A_{L,m}$ in \eqref{alm} explicitly. We find that
\begin{equation}
A_{L,m} = \frac{Lm}{8}+\frac{1}{2}\frac{\Gamma(mL)}{\Gamma\left(\frac{mL}{2}\right)^{2}}\,2^{-mL}. \label{almform}
\end{equation}
Inserting \eqref{almform} into \eqref{logNlemma} and using the duplication formula for the Gamma function completes the proof.
\end{proof}

\section{Laplace asymptotics}
\label{sec:laplace}
The purpose of this section is to prove the three Propositions \ref{prop:uniformweight}, \ref{prop:trunco} and \ref{prop:fint}.
	
	\begin{proof}[Proof of Proposition \ref{prop:uniformweight}]
	After the scaling $y_{j} \to y_{j}x^{\frac{1}{m}}$ the integral in \eqref{weight} takes the following form:
	\begin{equation}
	I := \int_{A_{x}}e^{K\phi(\vec{y})}h(\vec{y})\,d\vec{y} \label{ilxsaddle}
	\end{equation}
	where $K = \frac{L}{2}-1$, $\vec{y} = (y_{1},\ldots,y_{m-1})$, $d\vec{y}= dy_{1}\ldots dy_{m-1}$, $h(\vec{y}) = (y_{1}\ldots y_{m-1})^{-1}$,
	\begin{equation}
	A_{x} = \bigg\{\vec{y} \in [0,x^{-\frac{1}{m}}]^{m-1} : x^{\frac{1}{m}} \leq \prod_{i=1}^{m-1}y_{i} \leq x^{-(1-\frac{1}{m})}\bigg\},
	\end{equation}
	and the action, which depends implicitly on $x$ throughout the proof, is 
	\begin{equation}
	\phi(\vec{y}) = \log\left(1-\frac{x^{\frac{2}{m}}}{y_{1}^{2}\ldots y_{m-1}^{2}}\right)+\sum_{j=1}^{m-1}\log(1-x^{\frac{2}{m}}y_{j}^{2}). \label{action}
	\end{equation}
	Although the asymptotic analysis of \eqref{ilxsaddle} is standard, special care has to be taken if $x$ depends on $K$ such that $x \to 0$ as $K \to \infty$. One reason is that $\phi(\vec{y})$ vanishes as $x \to 0$ and this occurs at a rate $x^{\frac{2}{m}}$. As we shall see, provided that the \textit{combined quantity} $\eta := Kx^{\frac{2}{m}}$ is large, the Laplace method is still applicable, for which we now give the details.
	
	It is straightforward to check that the only critical point of \eqref{action} is $\vec{y} = (1,\ldots,1)$. Thus for $\epsilon>0$ small, we split $A_{x}$ into the region $E := [1-\epsilon,1+\epsilon]^{m-1}$ and its compliment $A_{x} \setminus E$ and denote the corresponding integrals $I_{E}$ and $I_{E^{\mathrm{c}}}$ respectively. Note that because $x < 1-\kappa$, no matter how small we choose $\kappa>0$, we can make $\epsilon>0$ small enough such that $E$ is contained strictly inside $A_{x}$. For $I_{E}$ we will Taylor expand the action at the point $\vec{y} = \vec{1}$ up to the fourth derivative:
	\begin{equation}
	\phi(\vec{y}) = \phi(\vec{1})+\frac{1}{2}(\vec{y}-\vec{1})^{\mathrm{T}}H(\vec{y}-\vec{1}) + E_{3}(\vec{y}-\vec{1})+R_{4}(\vec{y}) \label{taylorphi}
	\end{equation}
	where $H$ is the Hessian matrix of $\phi$ at $\vec{1}$, $E_{3}(\vec{y}-\vec{1})$ is the third order term in the Taylor expansion of $\phi$, and $R_{4}$ is the remainder term that can be bounded in terms of fourth order partial derivatives of $\phi$. 
	
	We will denote in what follows $\xi := K(E_{3}(\vec{y}-\vec{1})+R_{4}(\vec{y}))$. A direct computation of the third and fourth derivatives of $\phi$ shows that they are uniformly bounded on the set $E$ by a factor $x^{\frac{2}{m}}$ times an absolute constant depending only on $m$ and $\kappa$. Therefore we have the following bounds for $\vec{y} \in E$,
	\begin{align}
	\xi(\vec{y}) &= KE_{3}(\vec{y}) + O(\eta|\vec{y}-\vec{1}|^{4}),\label{xi1}\\
	\xi(\vec{y}) &= O(\eta|\vec{y}-\vec{1}|^{3}), \label{xi2}\\
	\xi(\vec{y}) &= O(\eta \epsilon |\vec{y}-\vec{1}|^{2}). \label{xi3}
	\end{align}
The Hessian matrix $H$ can be computed explicitly:
	\begin{equation}
	H_{ij} = -\frac{4x^{\frac{2}{m}}}{(1-x^{\frac{2}{m}})^{2}}\times \begin{cases} 2, & i=j\\
	1, & i\neq j \end{cases} \label{matrixHij}
	\end{equation}
	where $i,j=1,\ldots,m-1$. We have
	\begin{equation}
	\det(-H) = \frac{4^{m-1}x^{\frac{2(m-1)}{m}}}{(1-x^{\frac{2}{m}})^{2(m-1)}}\,m.
	\end{equation}
	
	To approximate $I_{E}$, we note that on $E$ the functions $\xi$ and $h$ are close to zero and $1$ respectively, and make use of the trivial identity
	\begin{equation}
	e^{\xi}h = 1 + (h-1) + (h-1)\xi + \xi +(e^{\xi}-1-\xi)h.
	\end{equation}
	Then we can decompose $I_{E} = e^{K\phi(\vec{1})}(I_{1}-I_{2}+I_{3}+I_{4}+I_{5}+I_{6})$, where after the shift $y_{j} \to y_{j}+1$ we have
	\begin{alignat}{3}
	&I_{1} = \int_{\mathbb{R}^{m-1}}e^{\frac{1}{2}K\vec{y}^{\mathrm{T}}H\vec{y}}d\vec{y}
	,\quad \,\,\, && I_{2} = \int_{\mathbb{R}^{m-1}\setminus[-\epsilon,\epsilon]^{m-1}}e^{\frac{1}{2}K\vec{y}^{\mathrm{T}}H\vec{y}}d\vec{y},\\
	&I_{3} = \int_{[-\epsilon,\epsilon]^{m-1}}e^{\frac{1}{2}K\vec{y}^{\mathrm{T}}H\vec{y}}(\tilde{h}(\vec{y})-1)\,d\vec{y}, \quad \,\,\, && I_{4} = \int_{[-\epsilon,\epsilon]^{m-1}}e^{\frac{1}{2}K\vec{y}^{\mathrm{T}}H\vec{y}}(\tilde{h}(\vec{y})-1)\tilde{\xi}(\vec{y})\,d\vec{y},\\
	&I_{5} =  \int_{[-\epsilon,\epsilon]^{m-1}}e^{\frac{1}{2}K\vec{y}^{\mathrm{T}}H\vec{y}}\tilde{\xi}(\vec{y})\,d\vec{y}, \quad \,\,\, 
	\end{alignat}
	and
	\begin{equation}
  I_{6} = \int_{[-\epsilon,\epsilon]^{m-1}}e^{\frac{1}{2}K\vec{y}^{\mathrm{T}}H\vec{y}}\left(e^{\tilde{\xi}(\vec{y})}-(1+\tilde{\xi}(\vec{y}))\right)\tilde{h}(\vec{y})\,d\vec{y},
  \end{equation}
where $\tilde{h}(\vec{y}) = h(\vec{y}+\vec{1})$ and $\tilde{\xi}(\vec{y}) = \xi(\vec{y}+\vec{1})$. The integral $I_{1}$ gives the leading order term:
	\begin{equation}
	\begin{split}
	e^{K\phi(\vec{1})}I_{1} &= (2\pi)^{\frac{m-1}{2}}\frac{e^{K\phi(\vec{1})}}{K^{\frac{m-1}{2}}\det(-H)^{\frac{1}{2}}}\\
	&= \frac{1}{\sqrt{m}}\,(2\pi)^{\frac{m-1}{2}}2^{-m+1}(1-x^{\frac{2}{m}})^{Km+m-1}\eta^{-\frac{m-1}{2}},
	\end{split}
	\end{equation}
where we remind that $\eta := Kx^{\frac{2}{m}}$. We must show that the relative error produced by the other terms $\{I_{j}\}_{j=2}^{6}$ is no larger than the error term claimed in \eqref{uniformweight}, namely that $I_{j}/I_{1} = O(\eta^{-1})$ for each $j=2,3,4,5,6$. The smallest eigenvalue of $-H$ will be helpful to prove this and we denote it $\Gamma_{1}$. From \eqref{matrixHij} we have that there is an absolute constant $c_{m}>0$ depending only on $m$ such that
	\begin{equation}
	\Gamma_{1} = c_{m}\frac{x^{\frac{2}{m}}}{(1-x^{\frac{2}{m}})^{2}}.
	\end{equation}
Then a standard diagonalisation of $H$ followed by Cauchy-Schwarz gives the bound
	\begin{equation}
	e^{\frac{1}{2}K\vec{y}^{\mathrm{T}}H\vec{y}} \leq e^{-\frac{1}{2}K\Gamma_{1}|\vec{y}|^{2}} \leq e^{-c \eta |\vec{y}|^{2}} \label{hessbound}
	\end{equation}
for some constant $c>0$. Applying this to $I_{2}$, we upper bound by extending all integrations to $\mathbb{R}$ except $y_{1}$. These $m-2$ integrations are explicit Gaussian integrals and yield
	\begin{align}
	I_{2} &\leq C\eta^{-\frac{m-2}{2}}\int_{\epsilon}^{\infty}e^{-c\eta y_{1}^{2}}\,dy_{1} \leq Ce^{-c\eta},
	\end{align}
which implies $I_{2}/I_{1} = O(e^{-c\eta})$. 

The estimation of $I_{3},I_{4},I_{5}$ and $I_{6}$ follows from Taylor expanding $\tilde{h}(\vec{y})$ and applying the estimates on $\xi$ in \eqref{xi1}, \eqref{xi2} and \eqref{xi3}. The degree of the monomials in $\vec{y}$ occurring in these expansions will give a corresponding power of $\eta^{-\frac{1}{2}}$ after integration. For example, in $I_{3}$ we Taylor expand $\tilde{h}(\vec{y})$ near $0$ to quadratic order, the linear terms produced giving zero in the integration by symmetry. In the error term of quadratic order, we upper bound by replacing the integration with $\mathbb{R}^{m-1}$, use \eqref{hessbound} and evaluate the resulting Gaussian integral giving $I_{3}/I_{1} = O(\eta^{-1})$. 

In $I_{4}$ and $I_{5}$ we apply the same approach, using $|\tilde{h}(\vec{y})-1||\xi(y)| = O(\eta |\vec{y}|^{4})$, which upon integration shows that $I_{4}/I_{1} = O(\eta^{-1})$. Then in $I_{5}$ the third order terms contained in $\tilde{\xi}(\vec{y})$ vanish in the integration by symmetry. The fourth order remainder term gives a contribution of order $O(\eta |\vec{y}|^{4})$. Thus $I_{5}/I_{1} = O(\eta^{-1})$. In $I_{6}$ we apply the bound $|e^{\tilde{\xi}}-1-\tilde{\xi}| \leq \frac{|\tilde{\xi}|^{2}}{2}e^{|\tilde{\xi}|} = O(\eta^{2}|\vec{y}|^{6}e^{\eta \epsilon C |\vec{y}|^{2}})$, where in the last estimate we applied \eqref{xi2} and \eqref{xi3}. This implies that
\begin{equation}
I_{6} \leq C\int_{\mathbb{R}^{m-1}}\,e^{-\eta |\vec{y}|^{2}(c-C \epsilon)}\eta^{2}|\vec{y}|^{6}d\vec{y},
\end{equation}
and choosing $\epsilon>0$ small enough to ensure $c-C\epsilon > 0$, we also have $I_{6}/I_{1} = O(\eta^{-1})$.
	
	Finally we treat the contribution from the complimentary region $E^{\mathrm{c}} = A_{x} \setminus [1-\epsilon,1+\epsilon]^{m-1}$ to the integral \eqref{ilxsaddle}. Without loss of generality suppose that $|y_{1}-1| > \epsilon$. On this region we will use the bound $1+x \leq e^{x}$, which implies that
	\begin{equation}
	\left(\left(\frac{1-x^{\frac{2}{m}}(y_{1}^{2}\ldots y_{m-1}^{2})^{-1}}{1-x^{\frac{2}{m}}}\right)\prod_{j=1}^{m-1}\left(\frac{1-x^{\frac{2}{m}}y_{j}^{2}}{1-x^{\frac{2}{m}}}\right)\right)^{K} \leq \mathrm{exp}\left(-K\frac{x^{\frac{2}{m}}}{1-x^{\frac{2}{m}}}\psi(\vec{y})\right)
	\end{equation}
	where the action $\psi(\vec{y})$ is independent of $x$ and is given by
	\begin{equation}
	\psi(\vec{y}) = \sum_{j=1}^{m-1}y_{j}^{2}+\frac{1}{y_{1}^{2}\ldots y_{m-1}^{2}}-m.
	\end{equation}
	The function $\psi(\vec{y})$ has a unique global minimum at $\vec{y} = \vec{1}$ and $\psi(\vec{1}) = 0$. The integration domain avoids the point $\vec{1}$ and so there exists a constant $c_{\epsilon}>0$ such that $\psi(\vec{y}) > c_{\epsilon}$ for all $\vec{y} \in E^{\mathrm{c}}$. We thus obtain the bound
	\begin{align}
	(1-x^{\frac{2}{m}})^{-Km}I_{E^{\mathrm{c}}} &\leq \int_{\mathbb{R}^{m-2}}\int_{1+\epsilon}^{\infty}\mathrm{exp}\left(-c\eta\,\psi(\vec{y})\right)\frac{d\vec{y}}{\vec{y}} \label{trivbndw}\\
	&\leq Ce^{-c\eta}.
	\end{align}
	This implies that $I_{E^{\mathrm{c}}}/I_{1} = O(e^{-c \eta})$. Repeating the estimate \eqref{trivbndw} on the full integration interval gives the crude estimate \eqref{crudeweight}.
	
	Therefore the biggest error terms came from $I_{3},I_{4},I_{5}$ and $I_{6}$ and had a relative error with $I_{1}$ of order $O(\eta^{-1})$, as required.
	\end{proof}

\begin{proof}[Proof of Proposition \ref{prop:trunco}]
We recall from the hypothesis of Theorem \ref{th:mainthm} that $L := L_{N}$ is a sequence of positive integers such that $\gamma := L_{N}/N \to \tilde{\gamma}$ as $N \to \infty$ with $\tilde{\gamma}>0$. We also recall that $\alpha := \frac{1}{1+\gamma}$. It will be important in what follows that due to the hypothesis on $\gamma$, we can find $N_{0} \in \mathbb{N}$ and $\epsilon>0$ independent of $N$ such that $\alpha < 1-\epsilon$ for all $N > N_{0}$. 

We start from the integral representation
\begin{equation}
\binom{L+n}{n} = \frac{1}{2\pi i}\oint_{C}(1-z)^{-L-1}z^{-n-1}\,dz, \label{intrepL}
\end{equation}
where $C$ is any loop enclosing the origin. Introducing $z^{(m)} := z_{1}\ldots z_{m}$ and $d\vec{z} = dz_{1}\ldots dz_{m}$, this implies that
\begin{align}
&f_{N-2,L}(x) = \frac{1}{(2\pi i)^{m}}\oint_{C^{m}}\left(\prod_{j=1}^{m}(1-z_{j})^{-L-1}\right)\sum_{n=0}^{N-2}\left(\frac{x}{z^{(m)}}\right)^{n}\,\frac{d\vec{z}}{z^{(m)}}\\
&= \frac{1}{(2\pi i)^{m}}\oint_{C^{m}}\left(\prod_{j=1}^{m}(1-z_{j})^{-L-1}\right)\,\frac{d\vec{z}}{z^{(m)}-x} \label{1sttermgeom}\\
&  - \frac{x^{N-1}}{(2\pi i)^{m}}\oint_{C^{m}}\left(\prod_{j=1}^{m}(1-z_{j})^{-L-1}z_{j}^{-N+1}\right)\,\frac{d\vec{z}}{z^{(m)}-x}. \label{geomsplit}
\end{align}
Let us start with the exterior region defined by $(\alpha+\omega)^{m} < |x| \leq 1$ with $\alpha = \frac{1}{1+\gamma}$. We let $C$ be the circle around the origin of radius $\alpha$, noting that our hypothesis on $L_{N}$ implies that for $N>N_{0}$ the contour $C$ lies strictly inside the unit circle. This is fundamental and without this assumption the asymptotics of $f_{N-2,L}$ and Theorem \ref{th:mainthm} can take a different form, indeed recall Theorem \ref{thm:weakregime}. 

Consider the $z_{m}$ integral in the first term of \eqref{geomsplit}. The integrand is analytic inside the unit circle except at the point $z_{m} = x/(z_{1}\ldots z_{m-1})$, which by the choice of contour and condition on $x$ lies outside $C$. Therefore by Cauchy's theorem the first term \eqref{1sttermgeom} is zero and we have
\begin{align}
f_{N-2,L}(x) &= - \frac{x^{N-1}}{(2\pi i)^{m}}\oint_{C^{m}}\left(\prod_{j=1}^{m}(1-z_{j})^{-\gamma N-1}z_{j}^{-N+1}\right)\,\frac{1}{z^{(m)}-x}\,d\vec{z} \\
&= - \frac{x^{N-1}}{(2\pi i)^{m}}\oint_{C^{m}}\prod_{j=1}^{m}e^{-N\phi(z_{j})}h_{x}(\vec{z})\,d\vec{z} \label{fnoutside}
\end{align}
where $\phi(z) := \gamma \log(1-z)+\log(z)$ and 
\begin{equation}
h_{x}(\vec{z}) := \left(\prod_{j=1}^{m}\frac{z_{j}}{1-z_{j}}\right)\,\frac{1}{z^{(m)}-x}. \label{fxdef}
\end{equation}
The integral \eqref{fnoutside} can now be analysed with the Laplace method. We denote $z_{j} = \alpha e^{i\theta_{j}}$ with $\theta_{j} \in (-\pi,\pi]$ for each $j=1,\ldots,m$. A simple computation shows that $\phi$ is stationary at $\vec{\theta} = (0,\ldots,0)$ where it achieves a global minimum. Now consider the region $E := [-\epsilon,\epsilon]^{m}$ and denote the contribution to the integral \eqref{fnoutside} as $I_{E}$ with the complimentary region denoted $I_{E^{\mathrm{c}}}$ where $E^{\mathrm{c}} := (-\pi,\pi]^{m}\setminus E$. On $E$ we Taylor expand the action
\begin{equation}
\phi(\alpha e^{i\theta}) = \gamma\log\left(\frac{\gamma}{1+\gamma}\right)+\log\left(\frac{1}{1+\gamma}\right)+\frac{1}{2}\frac{1+\gamma}{\gamma}\theta^{2}+O(\theta^{3}), \qquad |\theta| < \epsilon. \label{fnactiontaylor}
\end{equation}
Regarding the function $h_{x}(\vec{z})$ in \eqref{fxdef}, we have
\begin{equation}
\frac{1}{z^{(m)}-x} = \frac{1}{\alpha^{m}-x}\left(1+ \alpha^{m}\frac{1-e^{i\theta^{(m)}}}{\alpha^{m}e^{i\theta^{(m)}}-x}\right) \label{cauchydecomp}
\end{equation}
where $\theta^{(m)} = \theta_{1}+\ldots+\theta_{m}$.
As the distance between $\alpha^{m}$ and $x$ is at least $(\alpha+\omega)^{m}-\alpha^{m} \geq m\alpha^{m-1}\omega$, the second term in the brackets in \eqref{cauchydecomp} is $O(|\theta^{(m)}|\omega^{-1})$. The linear dependence on $\theta_{j}$'s will contribute an additional factor $N^{-\frac{1}{2}}$ in the Laplace method, so that this correction will contribute at order $O\left(\frac{1}{\sqrt{N}\omega}\right)$. Taylor expanding the other terms in \eqref{cauchydecomp}, we see that
\begin{equation}
h_{x}(\vec{z})\,d\vec{z} = i^{m}\,\frac{(\gamma(1+\gamma))^{-m}}{\alpha^{m}-x}\left(1 + O(\vec{\theta}\omega^{-1})\right)\,d\vec{\theta} \label{fxexpand}.
\end{equation}
Using \eqref{fxexpand} and \eqref{fnactiontaylor} we obtain the approximation
\begin{align}
I_{E} &:= \frac{(\gamma(1+\gamma))^{-m}}{(2\pi)^{m}}\left(\frac{1+\gamma}{\gamma}\right)^{m\gamma N}\frac{(1+\gamma)^{mN}}{\alpha^{m}-x}\left(\int_{[-\epsilon,\epsilon]}e^{-N\frac{1+\gamma}{\gamma}\theta^{2}}\,d\theta\right)^{m}\left(1+ O\left(\frac{1}{\sqrt{N}\omega}\right)\right) \label{fNerror}\\
&=\frac{e_{N,m}}{\alpha^{m}-x}\,\left(1+O(e^{-cN})+O\left(\frac{1}{\sqrt{N}\omega}\right)\right) .
\end{align}
On the other hand, the integral $I_{E^{\mathrm{c}}}$ is exponentially small relative to $I_{E}$. To see this, note that for $z=\alpha e^{i\theta}$ the action satisfies $\mathrm{Re}(\phi(z)) = \gamma\log|1-\alpha e^{i\theta}|+\log(\alpha)$ and that this function is monotonically increasing in $\theta$ with the minimum obtained at $\theta=0$ with $\mathrm{Re}(\phi(\alpha)) = \phi(\alpha)$. The monotonicity follows immediately from the formula $|1-\alpha e^{i\theta}|^{2} = (1-\alpha)^{2}+2\alpha(1-\cos(\theta))$ and corresponding monotonicity of the log and square root functions. Hence there is a constant $c_{\epsilon}>0$ such that $I_{E^{\mathrm{c}}}/I_{E} = O\left(\frac{e^{-c_{\epsilon}N}}{\omega}\right)$. This completes the proof of the Lemma in the exterior region $(\alpha+\omega)^{m} < |x| < 1-\kappa$.

For the region $-(\alpha+\omega)^{m} < x < (\alpha-\omega)^{m}$ we follow the steps in \eqref{geomsplit} but we use the big circle
\begin{equation}
C_{R} = \bigg\{-R\alpha+(R+1)\alpha e^{i\theta} : -\pi < \theta \leq \pi\bigg\}.
\end{equation}
In this case we will use the following facts proved in Lemma \ref{resbound}: Choosing $R$ large enough, in the integration over $z_{m}$ the pole at $x/(z_{1}\ldots z_{m-1})$ is strictly inside $C_{R}$. Furthermore, there exist absolute constants $\delta_{0}>0$ and $c>0$ such that
\begin{equation}
\inf_{x \in [-(\alpha+\delta_{0})^{m},(\alpha-\omega)^{m}]}\inf_{\theta_{k} \in [-\pi,\pi], k=1,\ldots,m}|z^{(m)}-x| > c\omega>0. \label{infbnd}
\end{equation}
Hence by Cauchy's integral formula we get
\begin{align}
f_{N-2,L}(x) &= \frac{1}{(2\pi i)^{m-1}}\oint_{C_{R}^{m-1}}\left(1-\frac{x}{z_{1}\ldots z_{m-1}}\right)^{-L-1}\prod_{j=1}^{m-1}(1-z_{j})^{-L-1}\,\frac{d\vec{z}}{z_{1}\ldots z_{m-1}} \label{firstterm2}\\
& - \frac{x^{N-1}}{(2\pi i)^{m}}\oint_{C_{R}^{m}}\left(\prod_{j=1}^{m}(1-z_{j})^{-L-1}z_{j}^{-N+1}\right)\,\frac{1}{z^{(m)}-x}\,d\vec{z} \label{secondterm2}\\
&= f_{\infty,L}(x) - \frac{x^{N-1}}{(2\pi i)^{m}}\oint_{C_{R}^{m}}\left(\prod_{j=1}^{m}(1-z_{j})^{-L-1}z_{j}^{-N+1}\right)\,\frac{1}{z^{(m)}-x}\,d\vec{z}, \label{secondterm3}
\end{align}
where the identification of the first integral \eqref{firstterm2} in terms of $f_{\infty,L}(x)$ follows from repeating the steps that led to \eqref{geomsplit} with $N=\infty$ and again applying Cauchy's integral formula. The estimation of the second term in \eqref{secondterm3} can be obtained by repeating the same steps that led to \eqref{fNerror}, using \eqref{cauchydecomp} and \eqref{infbnd}. In the complimentary region $E^{\mathrm{c}}$ we have $\mathrm{Re}(\phi(z)) = \gamma\log|1+R\alpha-(R+1)\alpha e^{i\theta}|+\log|-R\alpha + (R+1)\alpha e^{i\theta}|$ which again is an increasing function of $\theta$ with a unique minimum at $\theta=0$. This is easy to see from the formulas
\begin{align}
|-R\alpha+(R+1)\alpha e^{i\theta}|^{2} &= \alpha^{2}(1+2R(R+1)(1-\cos(\theta)),\\
|1+R\alpha - (R+1)\alpha e^{i\theta}|^{2} &= (1-\alpha)^{2}+2(R\alpha(1+(\alpha+1)R)+\alpha)(1-\cos(\theta)),
\end{align}
so that we again have a constant $c_{\epsilon}>0$ such that $I_{E^{\mathrm{c}}}/I_{E} = O\left(\frac{e^{-c_{\epsilon}N}}{\omega}\right)$. This concludes the proof of \eqref{tfintest} when $-(\alpha+\delta_{0})^{m} < x < (\alpha-\omega)^{m}$.
\end{proof}

\begin{proof}[Proof of Proposition \ref{prop:fint}]
Consider the integral representation for $f_{\infty,L}(x)$ given in \eqref{firstterm2}. For each $j=1,\ldots,m-1$, we scale $z_{j} \to x^{\frac{1}{m}}z_{j}$ and deform the contour to the unit circle $|z_{j}|=1$, setting $z_{j} = e^{i\theta_{j}}$ with $\theta_{j} \in [-\pi,\pi)$. This gives
\begin{equation}
f_{\infty,L}(x) = \frac{1}{(2\pi)^{m-1}}\int_{[-\pi,\pi]^{m-1}}e^{-(L+1)\phi(\vec{\theta})}\,d\vec{\theta}
\end{equation}
where $\vec{\theta} = (\theta_1,\ldots,\theta_{m-1})$, $d\vec{\theta} = d\theta_{1}\ldots d\theta_{m-1}$ and the action is given by
\begin{equation}
\phi(\vec{\theta}) = \log\left(1-x^{\frac{1}{m}}e^{-i(\theta_{1}+\ldots+\theta_{m-1})}\right)+\sum_{j=1}^{m-1}\log\left(1-x^{\frac{1}{m}}e^{i\theta_{j}}\right). \label{fintaction}
\end{equation}
It is straightforward to see that $\vec{\theta} = (0,\ldots,0)$ is a critical point of \eqref{fintaction} and this turns out to give the main contribution. There are also $m-1$ other critical points equally spaced on the unit circle, but their contribution will be exponentially suppressed compared to the critical point at the origin. Therefore we consider a sufficiently small $\epsilon>0$ and focus on $E = [-\epsilon,\epsilon]^{m-1}$, treating $E^{\mathrm{c}} = [-\pi,\pi]^{m-1} \setminus E$ separately. The Hessian matrix of $\phi$ is given by
\begin{equation}
H_{ij} = \frac{x^{\frac{1}{m}}}{(1-x^{\frac{1}{m}})^{2}}\times \begin{cases} 2, & i=j\\
1, & i\neq j \end{cases}. \label{matrixHijfint}
\end{equation}
On the set $E$, the considerations now follow the proof of Theorem \ref{weight} identically, taking now $\eta = Lx^{\frac{1}{m}}$, and we omit the details. On the complimentary region we aim to show that $I_{E}/I_{E^{\mathrm{c}}} = O(e^{-c\eta})$. Consider the case $m$ even as the case $m$ odd follows an identical pattern. We split the interval
\begin{equation}
[-\pi,\pi]^{m-1} = \bigcup\limits_{j=-\frac{m-2}{2}}^{\frac{m-2}{2}}A_{j}, \qquad A_{j} := \bigg\{\vec{\theta} \in [-\pi,\pi]^{m-1} : \sum_{j=1}^{m-1}\theta_{j} \in ((2j-1)\pi,(2j+1)\pi]\bigg\}.
\end{equation}
Denoting $\mathcal{A} = \frac{2x^{\frac{1}{m}}}{(1-x^{\frac{1}{m}})^{2}}$, we have the bound, for any $\theta \in (-\pi,\pi]$,
\begin{equation}
\begin{split}
\bigg{|}\frac{1-x^{\frac{1}{m}}}{1-x^{\frac{1}{m}}e^{i\theta}}\bigg{|}^{2} &= \frac{1}{1+\mathcal{A}(1-\cos(\theta))} \leq e^{-\frac{\mathcal{A}(1-\cos(\theta))}{1+2\mathcal{A}}} \leq e^{-c\,\frac{2\mathcal{A}}{1+2\mathcal{A}}\,\theta^{2}}\\
&= e^{-c\,\frac{4x^{\frac{1}{m}}}{(1+x^{\frac{1}{m}})^{2}}\,\theta^{2}}\leq e^{-c\,x^{\frac{1}{m}}\,\theta^{2}}
\end{split}
\end{equation}
where we used that $-\log(1+x) \leq -x(1+x)^{-1}$. Then if $\vec{\theta} \in A_{j}$, we have
\begin{equation}
\bigg{|}\frac{1-x^{\frac{1}{m}}}{1-x^{\frac{1}{m}}e^{-i(\theta_1+\ldots+\theta_{m-1})}}\bigg{|}^{2} \leq e^{-c\,x^{\frac{1}{m}}\,(\theta_{1}+\ldots+\theta_{m-1}-2j\pi)^{2}}
\end{equation}
and 
\begin{equation}
I_{E^{\mathrm{c}}} \leq (1-x^{\frac{1}{m}})^{-Lm}\sum_{j=-\frac{m-2}{2}}^{\frac{m-2}{2}}\int_{A_{j}\cap E^{\mathrm{c}}}e^{-c\,(L+1)x^{\frac{1}{m}}\,\psi(\vec{\theta})}\,d\vec{\theta}
\end{equation}
where $\psi(\vec{\theta}) = (\theta_{1}+\ldots+\theta_{m-1}-2j\pi)^{2}+\sum_{j=1}^{m-1}\theta_{j}^{2}$. It can easily be checked that the function $\psi(\vec{\theta})$ achieves its global minimum when $\theta_{k} = \frac{2j\pi}{m}$ for each $k=1,\ldots,m-1$, so that $\psi(\vec{\theta}) \geq \frac{(2j\pi)^{2}}{m}$. If $j=0$ we note that on $A_{j}\cap E^{\mathrm{c}}$ there is at least one $\theta_{k}$ satisfying $|\theta_{k}| > \epsilon$. Therefore, there exists a constant $c_{\epsilon}>0$ such that $\psi(\vec{\theta}) > c_{\epsilon}$ and $I_{E^{\mathrm{c}}} \leq (1-x^{\frac{1}{m}})^{-Lm}e^{-c_{\epsilon}Lx^{\frac{1}{m}}}$. This implies that $I_{E}/I_{\mathrm{E^{\mathrm{c}}}} = O(e^{-c\eta})$ as desired. Repeating these bounds with $E^{\mathrm{c}}$ replaced with the full integration interval $[-\pi,\pi]^{m-1}$ gives the crude bound \eqref{crudefint}. Furthermore, if $x<0$ we can again repeat these steps starting with \eqref{firstterm2}, and obtain
\begin{equation}
|f_{\infty}(x)| \leq  (1-(-x)^{\frac{1}{m}})^{-Lm}\sum_{j=-\frac{m-2}{2}}^{\frac{m-2}{2}}\int_{A_{j}}e^{-c\,(L+1)(-x)^{\frac{1}{m}}\,\tilde{\psi}(\vec{\theta})}\,d\vec{\theta} \label{negxbndpf}
\end{equation}
with the modified action $\tilde{\psi}(\vec{\theta}) = (\theta_{1}+\ldots+\theta_{m-1}-(2j+1)\pi)^{2}+\sum_{j=1}^{m-1}\theta_{j}^{2}$. The function $\tilde{\psi}(\vec{\theta})$ achieves its minimum at $\theta_{k} = \frac{(2j+1)\pi}{m}$ for $k=1,\ldots,m-1$ so that $\tilde{\psi}(\vec{\theta}) > \frac{((2j+1)\pi)^{2}}{m} > \pi^{2}/m$. Inserting this bound into \eqref{negxbndpf} completes the proof of \eqref{negxbnd}.
\end{proof}

\appendix

\section{The real Ginibre ensemble}
\label{sec:realgin}
Although we focus in this paper on the case of truncated orthogonal matrices, our approach applies almost without change to a similar model defined as follows. Let $X^{(m)} = G_{1}\ldots G_{m}$ denote the product of $m$ independent real Ginibre matrices of size $N \times N$, that is each matrix in the product is independent and consists of independent standard (real) Gaussian random variables. Like the truncated orthogonal matrices, they are an integrable model whose eigenvalues (including the real ones) form a Pfaffian point process \cite{IK14,FI16}. 

The analogues of exact formulae \eqref{iln} and \eqref{densiln} are known from the work of Forrester and Ipsen \cite{FI16}. It is assumed in \cite{FI16} that $N$ is even in this context, and we make the same assumption in what follows. There it is shown that
\begin{equation}
\mathbb{E}(N^{(m)}_{\mathbb{R}}) = \frac{N^{\frac{3m}{2}}}{(2\sqrt{2\pi})^{m}}\,\int_{\mathbb{R}}dx\,\int_{\mathbb{R}}dy\,|x-y|w_{\mathrm{Gin}}(N^{m/2}x)w_{\mathrm{Gin}}(N^{m/2}y)f_{N-2}(N^{m}xy)
\end{equation}
where
\begin{equation}
w_{\mathrm{Gin}}(N^{m/2}x) = \int_{\mathbb{R}^{m}}\mathrm{exp}\left(-\frac{1}{2}\sum_{j=1}^{m}\lambda_{j}^{2}\right)\delta(xN^{m/2}-\lambda_{1}\ldots \lambda_{m})\,d\lambda_{1}\ldots d\lambda_{m}, \label{ginweightdef}
\end{equation}
and
\begin{equation}
f_{N-2}(N^{m}x) = \sum_{n=0}^{N-2}\frac{(N^{m}x)^{n}}{(n!)^{m}}. \label{ginfdef}
\end{equation}

Using this formula, we can derive analogues of Theorems \ref{th:mainthm} and \ref{th:densthm} with an identical approach used in the present paper. 
\begin{theorem}
\label{th:ginthm}
Let $N^{(m)}_{\mathbb{R}}$ denote the number of real eigenvalues of $X^{(m)} = G_{1}\ldots G_{m}$. For any fixed $m \in \mathbb{N}$, we have
\begin{equation}
\mathbb{E}(N^{(m)}_{\mathbb{R}}) = \sqrt{\frac{2 N m}{\pi}} + O(1), \qquad N \to \infty. \label{o1errorgin}
\end{equation}
\end{theorem}
The leading term in Theorem \ref{th:ginthm} was first proved in \cite{S17} using an approach similar to the one used here in Section \ref{sec:weakregime}.

\begin{remark}
\label{rem:remainder}
When $m=1$ a precise asymptotic expansion is known for $\mathbb{E}(N_{\mathbb{R}})$, see \cite{EKS94}, and this shows that the $O(1)$ estimate in \eqref{o1errorgin} is optimal. Similarly, we believe that our $O(1)$ estimates in \eqref{o1errorgin} and \eqref{mainest} are optimal for any fixed $m \geq 1$. In \cite{S17} the estimate on the remainder in \eqref{o1errorgin} is at the less precise order $O(\log(N))$, hence Theorem \ref{th:ginthm} gives an improvement on this result.
\end{remark}

Now again from \cite{FI16} we have that the appropriately scaled and normalised density of real eigenvalues of the product matrix $N^{-\frac{m}{2}}G_{1}\ldots G_{m}$ is given by
\begin{equation}
\tilde{\rho}_{N}(x) = \frac{1}{\mathbb{E}{(N^{(m)}_{\mathbb{R}})}}\,\frac{N^{\frac{3m}{2}}}{(2\sqrt{2\pi})^{m}}\,\int_{\mathbb{R}}dy\,|x-y|w_{\mathrm{Gin}}(N^{m/2}x)w_{\mathrm{Gin}}(N^{m/2}y)f_{N-2}(N^{m}xy).
\end{equation}
\begin{theorem}
\label{th:gindensthm}
For any bounded continuous test function $h$, we have
\begin{equation}
\lim_{N \to \infty}\int_{\mathbb{R}}h(x)\rho_{N}(x)\,dx = \int_{\mathbb{R}}h(x)\rho(x)\,dx, \label{hlimitgin}
\end{equation}
where 
\begin{equation}
\rho(x) = \frac{1}{2m\,x^{1-\frac{1}{m}}}\mathbbm{1}_{x \in (-1,1)}.
\end{equation}
Furthermore, for any fixed $x \in \mathbb{R} \setminus \{-1,0,1\}$, we have the pointwise limit $\lim_{N \to \infty}\rho_{N}(x) = \rho(x)$.
\end{theorem}
The weak convergence \eqref{hlimitgin} was also established in \cite{S17}, proving a conjecture of Forrester and Ipsen in \cite{FI16}.

Let us now explain the small changes required to prove Theorems \ref{th:ginthm} and \ref{th:gindensthm} using the techniques of the present paper. First note that the definition of the weight \eqref{ginweightdef} can be written in the equivalent form
\begin{equation}
w_{\mathrm{Gin}}(N^{m/2}x) = 2^{m-1}\int_{[0,\infty)^{m-1}}\mathrm{exp}\left(-\frac{N|x|^{\frac{2}{m}}}{2}\left(\frac{1}{u_{1}^{2}\ldots u_{m-1}^{2}}+\sum_{j=1}^{m-1}u_{j}^{2}\right)\right)\,\frac{du_{1}}{u_1}\ldots \frac{du_{m-1}}{u_{m-1}}. \label{ginweightdef2}
\end{equation}
Then a saddle point analysis as in the proof of Proposition \ref{prop:uniformweight} gives the following result. 
\begin{proposition}
\label{prop:gin1}
Fix a large constant $M>0$. Then we have the following asymptotic estimate uniformly on $|x| \in [MN^{-\frac{m}{2}},\infty)$
\begin{equation}
w_{\mathrm{Gin}}(N^{m/2}x) = N^{-\frac{m-1}{2}}\,e^{-\frac{Nm}{2}x^{\frac{2}{m}}}\frac{(4\pi)^{\frac{m-1}{2}}}{\sqrt{m}}|x|^{-\frac{m-1}{m}}\left(1+O\left(\frac{1}{N|x|^{\frac{2}{m}}}\right)\right), \qquad N \to \infty. \label{wresult}
\end{equation}
\end{proposition}
In fact the proof of Proposition \ref{prop:gin1} is more straightforward than for Proposition \ref{prop:uniformweight}, because it is already immediate from \eqref{ginweightdef2} that $\eta := N|x|^{\frac{2}{m}}$ is the appropriate large parameter in the Laplace asymptotics.

Regarding $f_{N-2}(N^{m}x)$ in \eqref{ginfdef}, we mimic the proof of Proposition \ref{prop:trunco} almost identically. For instance, instead of \eqref{intrepL}, we use the integral representation
\begin{equation}
\frac{N^{n}}{n!} = \frac{1}{2\pi i}\oint_{C}e^{Nz}\,z^{-n-1}\,dz, \label{jfact}
\end{equation}
and follow all the steps identically, the main difference being that the saddle point is now located at $z=1$ instead of $z=\alpha$. This leads to
\begin{proposition}
\label{prop:gin2}
As $N \to \infty$ we have the following estimate uniformly on $x \in \mathbb{R}\setminus((1-\omega)^{m},(1+\omega)^{m})$,
\begin{equation}
f_{N-2}(N^{m}x) = f_{\infty}(N^{m}x)\mathbbm{1}_{-(1+\omega)^{m} < x < (1-\omega)^{m}} + \frac{x^{N-1}e^{mN}}{(2\pi N)^{m/2}(x-1)}\left(1+O\left(\frac{1}{\sqrt{N}\omega}\right)\right), \label{fintest}
\end{equation}
where we take $\omega = N^{-\frac{1}{2}}$.
\end{proposition}

Now regarding the function,
\begin{equation}
f_{\infty}(N^{m}x) = \sum_{n=0}^{\infty}\frac{(N^{m}x)^{n}}{(n!)^{m}}, \label{gininf}
\end{equation}
we insert the integral representation \eqref{jfact} and find that for $x>0$, we have
\begin{equation}
f_{\infty}(N^{m}x) = \frac{1}{(2\pi i)^{m-1}}\,\oint_{C^{m-1}}\mathrm{exp}\left(Nx^{\frac{1}{m}}\left(z_{1}+\ldots+z_{m-1}+\frac{1}{z_{1}\ldots z_{m-1}}\right)\right)\,\frac{dz_{1}}{z_{1}}\ldots \frac{dz_{m-1}}{z_{m-1}}. \label{gininfrep}
\end{equation}
The representation \eqref{gininfrep} is well-suited for applying the Laplace method with large parameter $\eta := Nx^{\frac{1}{m}}$, in complete analogy with the proof of Proposition \ref{prop:fint}.

\begin{proposition}
\label{prop:gin3}
Fix a large constant $M>0$. Then we have the following estimate uniformly on $x \in (MN^{-m},\infty)$,
\begin{equation}
f_{\infty}(N^{m}x) = (2\pi)^{-(m-1)/2}x^{-\frac{m-1}{2m}}e^{Nmx^{\frac{1}{m}}}N^{-\frac{m-1}{2}}\,\frac{1}{\sqrt{m}}\left(1+O\left(\frac{1}{x^{\frac{1}{m}}N}\right)\right), \qquad N \to \infty. \label{asyfint}
\end{equation}
\end{proposition}
\begin{remark}
Analogously to Remark \ref{rem:pointwise}, at the pointwise level the asymptotics \eqref{wresult} and \eqref{asyfint} were derived in the context of products of independent complex Ginibre matrices in \cite{AB12}. 
\end{remark}
\begin{remark}
We make a remark on the $m=1$ case of Proposition \ref{prop:gin2}. This corresponds to studying asymptotics of a truncated exponential function and has been considered by many authors. At the level of precision required in this paper, this appeared in \cite{BG07}, but more precise asymptotics are known. The proof given in \cite{BG07} directly motivated our proof of Propositions \ref{prop:trunco} and \ref{prop:gin2} that hold for any $m\geq 1$. 
\end{remark}
Now with Propositions \ref{prop:gin1}, \ref{prop:gin2} and \ref{prop:gin3}, the proofs in Sections \ref{sec:leading} and \ref{sec:proofs} go through in complete analogy and we omit the details. This is how we prove Theorems \ref{th:ginthm} and \ref{th:gindensthm}.

\section{Miscellaneous bounds}

\begin{lemma}
\label{lem:gaussbnd}
For any $(u,v) \in [0,1]^{2}$, we have the inequality
\begin{equation}
\frac{(1-u^{2})(1-v^{2})}{(1-uv)^{2}} \leq e^{-(u-v)^{2}}. \label{gaussbnd}
\end{equation}
\end{lemma}

\begin{proof}
Using the bound $1+x \leq e^{x}$ for all $x \in \mathbb{R}$ we have
\begin{equation}
\begin{split}
\frac{(1-u^{2})(1-v^{2})}{(1-uv)^{2}} &= 1-\frac{(u-v)^{2}}{(1-uv)^{2}} \leq e^{-\frac{(u-v)^{2}}{(1-uv)^{2}}}\leq e^{-(u-v)^{2}},
\end{split}
\end{equation}
where in the last inequality we used $\frac{1}{1-uv} > 1$ for all $(u,v) \in [0,1]^{2}$.
\end{proof}

\begin{lemma}
\label{lem:qbound}
Consider the function $Q : [0,\infty)^{2} \to [0,\infty)$ defined by
\begin{equation}
Q(u,v) = \frac{(u+v)^{m}-u^{m}}{(u(u+v))^{\frac{m-1}{2}}}. \label{qdef}
\end{equation}
Then for $\delta>0$ sufficiently small and $M>0$ arbitrary, we have for $(u,v) \in [0,M] \times [0,\delta]$ the uniform bound
\begin{equation}
Q(u,v) = mv + \sum_{j=0}^{m-1}O\left(\frac{v^{3+j}}{u^{2+j}}\right). \label{qsmallu}
\end{equation}
Furthermore, for any $\epsilon>0$, there is a constant $C_{\epsilon,M}>0$ independent of $u$ such that on the domain $(u,v) \in [\epsilon,M] \times [0,M]$ we have $Q(u,v) \leq C_{\epsilon,M}v$.
\end{lemma}

\begin{proof}
We have
\begin{equation}
Q(u,v) = u\frac{\left(1+\frac{v}{u}\right)^{m}-1}{\left(1+\frac{v}{u}\right)^{\frac{m-1}{2}}} = \left(1+\frac{v}{u}\right)^{-\frac{m-1}{2}}\sum_{j=0}^{m-1}\binom{m}{j+1}\frac{v^{j+1}}{u^{j}}. \label{qbin}
\end{equation}
This implies that $Q(u,v)\leq C_{\epsilon,M}v$ provided $u > \epsilon$. Taylor expanding near $v=0$, we have
\begin{equation}
\left(1+\frac{v}{u}\right)^{-\frac{m-1}{2}} = 1 - \frac{m-1}{2}\,\frac{v}{u} + O\left(\frac{v^{2}}{u^{2}}\right), \qquad 0 < v < \delta. \label{unifbigo}
\end{equation}
To obtain the uniform big-$O$ term in \eqref{unifbigo} note that 
\begin{equation}
\frac{d^{2}}{dv^{2}}\left(1+\frac{v}{u}\right)^{-\frac{m-1}{2}} = \frac{\left(1+\frac{v}{u}\right)^{-\frac{m-1}{2}}}{4(u+v)^{2}}(m^{2}-1) \leq \frac{m^{2}-1}{4u^{2}}.
\end{equation}
Then inserting \eqref{unifbigo} into \eqref{qbin} the term proportional to $\frac{v^{2}}{u}$ cancels and we obtain \eqref{qsmallu}.
\end{proof}

\begin{lemma}
\label{resbound}
Let the circles $C_{R}$ be defined by
\begin{equation}
z_{k} = -R\alpha+(R+1)\alpha e^{i\theta_{k}}, \qquad -\pi < \theta_{k} \leq \pi, \label{bigcirc}
\end{equation}
and define $z^{(m)} := z_{1}\ldots z_{m}$. Consider the parameters $\omega$ and $\alpha$ as in Proposition \ref{prop:trunco}. Then for $R>0$ large enough, there exist absolute constants $c>0$ and $\delta_{0}>0$ independent of $N$ such that
\begin{equation}
\label{appinfbnd}
\inf_{x \in [-(\alpha+\delta_{0})^{m},(\alpha-\omega)^{m}]}\inf_{\theta_{k} \in [-\pi,\pi], k=1,\ldots,m}|z^{(m)}-x| > c\omega.
\end{equation}
Furthermore, for any $x \in [-(\alpha+\delta_{0})^{m},(\alpha-\omega)^{m}]$ and for each $k=1,\ldots,m$, we have that $\frac{xz_{k}}{z^{(m)}}$ belongs to the interior of $C_{R}$.
\end{lemma}

\begin{proof}
We have
\begin{align}
|z_{k}|^{2} = 2\alpha^{2}R^{2}(1-\cos(\theta_{k}))+2R\alpha^{2}(1-\cos(\theta_{k}))+\alpha^{2} \label{radialform},
\end{align}
and clearly each $|z_{k}| \geq \alpha$. When $x>0$ we have
\begin{equation}
|z^{(m)}-x| \geq ||z^{(m)}|-|x|| \geq \alpha^{m}-(\alpha-\omega)^{m} \geq c\omega.
\end{equation}
We also have
\begin{equation}
\bigg{|}\frac{xz_{\ell}}{z^{(m)}}\bigg{|} \leq \frac{(\alpha-\omega)^{m}}{\alpha^{m-1}} \leq \alpha(1-\omega/\alpha)^{m} < \alpha,
\end{equation}
and therefore $\frac{xz_{\ell}}{z^{(m)}}$ belongs to the interior of $C_{R}$, as $C_{R}$ strictly includes any disc of radius smaller than $\alpha$. 

If $x\leq 0$, in particular when $x$ is close to $-\alpha^{m}$ we have to check that there is not too much winding in the product $z^{(m)}$ that might allow $\phi := \mathrm{Arg}(z^{(m)}) = \pi$. We start by supposing that for at least one $k=1,\ldots,m$, we have $|\theta_{k}| \geq \epsilon/R$. Then for $R$ large enough and any fixed $\epsilon>0$
\begin{equation}
1-\cos(\theta_{k}) \geq 1-\cos(\epsilon/R) \geq \left(\frac{\epsilon}{R}\right)^{2}\frac{1}{\pi}. \label{jordbnd}
\end{equation}
Hence, for this $k$, inserting \eqref{jordbnd} into \eqref{radialform} gives $|z_{k}|^{2} \geq \alpha^{2}(1+\epsilon^{2}/2)$. Using this and $|z_{j}| \geq \alpha$ shows that $|z^{(m)}| \geq \alpha^{m}(1+\epsilon^{2}/8)$. Furthermore $|x| < \alpha^{m}(1+c_{m}\delta_{0})$ for some $c_{m}>0$ depending only on $m$. Then
\begin{equation}
|z_{1}\ldots z_{m}-x| \geq ||z_{1}|\ldots |z_{m}|-|x|| \geq \alpha^{m}(\epsilon^{2}/8-c_{m}\delta_{0}) > c\omega, \label{radialbnd}
\end{equation}
which is valid for any $\epsilon > \sqrt{8}\sqrt{c_{m}\delta_{0}+\omega}$. Thus we see that by choosing $\delta_{0}$ small enough we can allow arbitrarily small $\epsilon>0$. The same bound shows that for any $\ell=1,\ldots,m$, we have
\begin{equation}
|xz_{\ell}/z^{(m)}| \leq \alpha\frac{1+c_{m}\delta_{0}}{1+\frac{\epsilon^{2}}{8}} < \alpha,
\end{equation}
which implies that $xz^{\ell}/z^{(m)}$ is strictly in the interior of $C_{R}$. This bounds the region where at least one of the $|\theta_{k}| \geq \epsilon/R$. The complement of this region is where $|\theta_{j}| < \epsilon/R$ for all $j=1,\ldots,m$. We claim that on this region
\begin{equation}
|\mathrm{Arg}(z_1 \ldots z_{m})| = |\mathrm{Arg}(z_1) + \ldots + \mathrm{Arg}(z_{m})| \leq 2\epsilon m. \label{argbound}
\end{equation}
To prove this we compute the argument of each $z_{j}$,
\begin{equation}
\mathrm{Arg}(z_{j}) = \mathrm{tan}^{-1}\left(\frac{\sin(\theta_{j})}{-\frac{R}{R+1}+\cos(\theta_{j})}\right).
\end{equation}
By symmetry it suffices to assume $\theta_{j} \geq 0$. Using $\sin(x) \leq x$, we have for the numerator $
\sin(\theta_{j}) \leq \theta_{j} \leq \epsilon/R$. The cosine inequality $\cos(x) \geq 1-2x/\pi$ gives for the denominator $-\frac{R}{R+1}+\cos(\theta_{j}) \geq  \frac{1}{R+1}-2\epsilon/(\pi R)$ and clearly $\frac{R}{R+1}-2\epsilon/\pi > \frac{1}{2}$. These bounds imply that 
\begin{equation}
\frac{\sin(\theta_{j})}{-\frac{R}{R+1}+\cos(\theta_{j})} \leq 2\epsilon.
\end{equation}
Now by monotonicity of $\mathrm{tan}^{-1}$ we have $\mathrm{Arg}(z_{j}) \leq 2\epsilon$ and this implies \eqref{argbound}. We thus have
\begin{align}
|z^{(m)}-x|^{2} = |z|^{2}+x^{2}-2|z|x\cos(\phi) \geq \alpha^{2m} > \omega,
\end{align}
which holds for example if $\phi := \mathrm{Arg}(z^{(m)}) < \pi/2$ (note that \eqref{argbound} implies $\phi < 2\epsilon m$). We also have
\begin{equation}
\begin{split}
|\mathrm{Arg}(z_{\ell}x/z^{(m)})-\pi| &\leq 2m\epsilon,\\
\bigg{|}\frac{z_{\ell}x}{z^{(m)}}\bigg{|} &\leq \alpha(1+\delta_{0})^{m},
\end{split}
\end{equation}
and choosing $\delta_{0}$ and $\epsilon$ small enough implies the strict inclusion of $z_{\ell}x/z^{(m)}$ inside the contour $C_{R}$.
\end{proof}

\section{Exact calculation of a multiple integral}
In this section we compute the integral from \eqref{alm} explicitly, namely the integral
\begin{equation}
A_{L,m} = \frac{1}{\Gamma\left(\frac{L}{2}\right)^{2m}}\int_{\mathbb{R}_{+}^{2m}}d\vec{t}\,d\vec{r}\,\prod_{\ell=1}^{m}t_{\ell}^{\frac{L}{2}-1}r_{\ell}^{\frac{L}{2}-1}e^{-t_{\ell}-r_{\ell}}\,\frac{t_{1}+\ldots+t_{m}}{2}\,\mathbbm{1}_{r_{1}+\ldots+r_{m} < t_{1}+\ldots+t_{m}}.
\end{equation}
\begin{lemma}
We have 
\begin{equation}
A_{L,m} = \frac{mL}{8}+\frac{1}{2}\frac{\Gamma(mL)}{\Gamma\left(\frac{mL}{2}\right)^{2}}\,2^{-mL}.
\end{equation}
\end{lemma}

\begin{proof}
By permutation invariance it is sufficient to consider the following
\begin{equation}
B_{L,m} := \frac{1}{\Gamma\left(\frac{L}{2}\right)^{2m}}\int_{\mathbb{R}_{+}^{2m}}d\vec{t}\,d\vec{r}\,\prod_{\ell=1}^{m}t_{\ell}^{\frac{L}{2}-1}r_{\ell}^{\frac{L}{2}-1}e^{-t_{\ell}-r_{\ell}}\,t_{m}\,\mathbbm{1}_{r_{1}+\ldots+r_{m} < t_{1}+\ldots+t_{m}},
\end{equation}
so that clearly $A_{L,m} = \frac{m}{2}B_{L,m}$. We would also prefer to replace the $t_{m}$ with $r_{m}$. Note that if we replaced $t_{m}$ with $t_{m}+r_{m}$, then by symmetry between the $r_{\ell}$ and $t_{\ell}$ variables we can drop the indicator function after multiplying by a factor $\frac{1}{2}$. The remaining integrals are explicit and we have
\begin{equation}
\frac{1}{\Gamma\left(\frac{L}{2}\right)^{2m}}\int_{\mathbb{R}_{+}^{2m}}d\vec{t}\,d\vec{r}\,\prod_{\ell=1}^{m}t_{\ell}^{\frac{L}{2}-1}r_{\ell}^{\frac{L}{2}-1}e^{-t_{\ell}-r_{\ell}}\,(t_{m}+r_{m})\,\mathbbm{1}_{r_{1}+\ldots+r_{m} < t_{1}+\ldots+t_{m}} = \frac{L}{2}.
\end{equation}
We thus have $B_{L,m} = \frac{L}{2}-F_{L,m}$ where
\begin{equation}
F_{L,m} := \frac{1}{\Gamma\left(\frac{L}{2}\right)^{2m}}\int_{\mathbb{R}_{+}^{2m}}d\vec{t}\,d\vec{r}\,\prod_{\ell=1}^{m}t_{\ell}^{\frac{L}{2}-1}r_{\ell}^{\frac{L}{2}-1}e^{-t_{\ell}-r_{\ell}}r_{m}\,\mathbbm{1}_{r_{1}+\ldots+r_{m} < t_{1}+\ldots+t_{m}}.
\end{equation}
To compute $F_{L,m}$, we will integrate by parts in the variable $r_{m}$, differentiating the factor $r_{m}^{\frac{L}{2}}$ and integrating $e^{-r_{m}}$. The second term in the integration by parts is thus completely explicit, again by symmetry, and gives a contribution $\frac{L}{4}$. Thus we have $F_{L,m} = \frac{L}{4}-D_{L,m}$ where $D_{L,m}$ is the boundary term of the integration by parts. To calculate this, we parameterise the integration such that $\vec{t} \in \mathbb{R}_{+}^{m}$ are unconstrained, while $\vec{r}$ satisfy the constraints $K_{m+1-\ell} := \{0 < r_{m+1-\ell} < t_{1}+\ldots+t_{m}-(r_{1}+\ldots+r_{m-\ell}\}$ where $\ell=1,\ldots,m$. Setting $r_{m} = t_{1}+\ldots+t_{m}-(r_{1}+\ldots+r_{m-1})$ results in the boundary term,
\begin{equation}
D_{L,m} = \frac{1}{\Gamma\left(\frac{L}{2}\right)^{2m}}\int_{\mathbb{R}_{+}^{m}}d\vec{t}\,
\prod_{\ell=1}^{m}t_{\ell}^{\frac{L}{2}-1}e^{-2t_{\ell}}\prod_{\ell=1}^{m-1}\int_{K_{\ell}}dr_{\ell}\,r_{\ell}^{\frac{L}{2}-1}(t_{1}+\ldots+t_{m}-(r_{1}+\ldots+r_{m-1}))^{\frac{L}{2}-1}.
\end{equation}
Then we integrate over each $r_{\ell}$ variable, starting with $r_{m-1}$, repeatedly using the formula
\begin{equation}
\int_{0}^{u}dr_{\ell}\,r_{\ell}^{a-1}(u-r_{\ell})^{b-1} = u^{a+b-1}\,\frac{\Gamma(a)\Gamma(b)}{\Gamma(a+b)}, \label{betaint}
\end{equation}
for each integral. We obtain
\begin{equation}
D_{L,m} =\frac{1}{\Gamma\left(\frac{L}{2}\right)^{2m}}\left(\prod_{\ell=1}^{m-1}\frac{\Gamma\left(\frac{\ell L}{2}+1\right)\Gamma\left(\frac{L}{2}\right)}{\Gamma\left(\frac{\ell L}{2}+\frac{L}{2}+1\right)}\right)\int_{\mathbb{R}_{+}^{m}}d\vec{t}\,
\prod_{\ell=1}^{m}t_{\ell}^{\frac{L}{2}-1}e^{-2t_{\ell}}(t_{1}+\ldots+t_{m})^{\frac{mL}{2}}.
\end{equation}
Now to compute the integrals over the $t_{\ell}$ variables we substitute the $t_{m}$ variable with $u = t_{1}+\ldots+t_{m}$ and parameterise such that $u \in \mathbb{R}_{+}$, while $t_{1},\ldots,t_{m-1}$ satisfy the constraints $H_{m-j} = \{0 < t_{m-j} < u-(t_{1}+\ldots+t_{m-j-1}\}$ where $j=1,\ldots,m-1$. The integrals over $H_{m-j}$ are computed again using \eqref{betaint} repeatedly and we find
\begin{align}
D_{L,m} &=\frac{1}{\Gamma\left(\frac{L}{2}\right)^{2m}}\left(\prod_{\ell=1}^{m-1}\frac{\Gamma\left(\frac{\ell L}{2}+1\right)\Gamma\left(\frac{L}{2}\right)}{\Gamma\left(\frac{\ell L}{2}+\frac{L}{2}+1\right)}\,\frac{\Gamma\left(\frac{\ell L}{2}\right)\Gamma\left(\frac{L}{2}\right)}{\Gamma\left(\frac{\ell L}{2}+\frac{L}{2}\right)}\right)\,\int_{0}^{\infty}du\,u^{mL-1}e^{-2u}\\
&= \frac{1}{\Gamma\left(\frac{L}{2}\right)^{2m}}\left(\prod_{\ell=1}^{m-1}\frac{\Gamma\left(\frac{\ell L}{2}+1\right)\Gamma\left(\frac{L}{2}\right)}{\Gamma\left(\frac{\ell L}{2}+\frac{L}{2}+1\right)}\,\frac{\Gamma\left(\frac{\ell L}{2}\right)\Gamma\left(\frac{L}{2}\right)}{\Gamma\left(\frac{\ell L}{2}+\frac{L}{2}\right)}\right)\,2^{-mL}\Gamma(mL) \label{prodterms}\\
&= \frac{\Gamma(mL)}{m\Gamma(\frac{mL}{2})^{2}}\,2^{-mL}, \label{dlmfinal}
\end{align}
where we exploited the cancellation of successive terms in the products \eqref{prodterms}. Substituting \eqref{dlmfinal} in the formula $A_{L,m} = \frac{mL}{8}+\frac{m}{2}\,D_{L,m}$ concludes the proof of the Lemma.
\end{proof}

\bibliographystyle{abbrv}
\bibliography{productsbib}

\begin{thebibliography}{10}

\bibitem{KNTK16}
K.~Adhikari, N.~Kishore~Reddy, T.~Ram~Reddy, and K.~Saha.
\newblock Determinantal point processes in the plane from products of random
  matrices.
\newblock {\em Ann. Inst. Henri Poincar\'{e} Probab. Stat.}, 52(1):16--46,
  2016.

\bibitem{AB12}
G.~Akemann and Z.~Burda.
\newblock Universal microscopic correlation functions for products of
  independent {G}inibre matrices.
\newblock {\em J. Phys. A}, 45(46):465201, 18, 2012.

\bibitem{ABK20}
G.~Akemann, Z.~Burda, and M.~Kieburg.
\newblock Universality of local spectral statistics of products of random
  matrices.
\newblock {\em Phys. Rev. E.}, 102:052134, 2020.

\bibitem{ABKN14}
G.~Akemann, Z.~Burda, M.~Kieburg, and T.~Nagao.
\newblock Universal microscopic correlation functions for products of truncated
  unitary matrices.
\newblock {\em J. Phys. A}, 47(25):255202, 26, 2014.

\bibitem{AI15}
G.~Akemann and J.~R. Ipsen.
\newblock Recent exact and asymptotic results for products of independent
  random matrices.
\newblock {\em Acta. Phys. Pol. B}, 46(9):1747--1784, 2015.

\bibitem{BB20}
J.~Baik and T.~Bothner.
\newblock The largest real eigenvalue in the real {G}inibre ensemble and its
  relation to the {Z}akharov-{S}habat system.
\newblock {\em Ann. Appl. Probab.}, 30(1):460--501, 2020.

\bibitem{BS09}
A.~Borodin and C.~D. Sinclair.
\newblock The {G}inibre ensemble of real random matrices and its scaling
  limits.
\newblock {\em Comm. Math. Phys.}, 291(1):177--224, 2009.

\bibitem{BL85}
P.~Bougerol and J.~Lacroix.
\newblock {\em Products of random matrices with applications to Schr\"odinger
  operators}, volume~8 of {\em Progress in probability and statistics (P. Huber
  and M. Rosenblatt, eds.)}.
\newblock Birkh\"auser, Boston, 1985.

\bibitem{BG07}
R.~Boyer and W.~M.~Y. Goh.
\newblock On the zero attractor of the {E}uler polynomials.
\newblock {\em Adv. in Appl. Math.}, 38(1):97--132, 2007.

\bibitem{BJW10}
Z.~Burda, R.~A. Janik, and B.~Waclaw.
\newblock Spectrum of the product of independent random {G}aussian matrices.
\newblock {\em Phys. Rev. E (3)}, 81(4):041132, 12, 2010.

\bibitem{BNS12}
Z.~Burda, M.~Nowak, and A.~Swiech.
\newblock New spectral relations between products and powers of isotropic
  random matrices.
\newblock {\em Phys. Rev. E.}, 86:061137, 2012.

\bibitem{CPV93}
A.~Crisanti, G.~Paladin, and A.~Vulpiani.
\newblock Products of random matrices.
\newblock In {\em Random Matrices and their applications}, volume 104 of {\em
  Springer series in Solid-State Sciences}. Springer-Verlag, Berlin,
  Heidelberg, 1993.

\bibitem{MPT16}
L.~C.~G. del Molino, K.~Pakdaman, and J.~Touboul.
\newblock Real eigenvalues of non-symmetric random matrices: {T}ransitions and
  {U}niversality.
\newblock eprint = \texttt{arXiv:1605.00623[math-ph]}.

\bibitem{EKS94}
A.~Edelman, E.~Kostlan, and M.~Shub.
\newblock How many eigenvalues of a random matrix are real?
\newblock {\em J. Amer. Math. Soc.}, 7(1):247--267, 1994.

\bibitem{FTZ20}
W.~FitzGerald, R.~Tribe, and O.~Zaboronski.
\newblock Sharp asymptotics for {F}redholm {P}faffians related to interacting
  particle systems and random matrices.
\newblock {\em Electron. J. Probab.}, 25:Paper No. 116, 15, 2020.

\bibitem{F10}
P.~J. Forrester.
\newblock The limiting {K}ac random polynomial and truncated random orthogonal
  matrices.
\newblock {\em J. Stat. Mech: Theo. Exper.}, P12018, 2010.

\bibitem{FI16}
P.~J. Forrester and J.~R. Ipsen.
\newblock Real eigenvalue statistics for products of asymmetric real {G}aussian
  matrices.
\newblock {\em Linear Algebra and its Applications}, 510:259--290, 2016.

\bibitem{FIK20}
P.~J. Forrester, J.~R. Ipsen, and S.~Kumar.
\newblock How many eigenvalues of a product of truncated orthogonal matrices
  are real?
\newblock {\em Exp. Math.}, 29(3):276--290, 2020.

\bibitem{FK18}
P.~J. Forrester and S.~Kumar.
\newblock The probability that all eigenvalues are real for products of
  truncated real orthogonal random matrices.
\newblock {\em J. Theoret. Probab.}, 31(4):2056--2071, 2018.

\bibitem{FN07}
P.~J. Forrester and T.~Nagao.
\newblock Eigenvalue {S}tatistics of the {R}eal {G}inibre {E}nsemble.
\newblock {\em Phys. Rev. Lett.}, 99:050603, 2007.

\bibitem{FK60}
H.~Furstenberg and H.~Kesten.
\newblock Products of random matrices.
\newblock {\em Ann. Math. Statist.}, 31:457--469, 1960.

\bibitem{FK15}
Y.~V. Fyodorov and B.~A. Khoruzhenko.
\newblock A {N}onlinear {A}nalogue of {M}ay-{W}igner {I}nstability
  {T}ransition.
\newblock {\em Proc. Natl. Acad. Sci. USA}, 113:6827--6832, 2016.

\bibitem{GP19}
M.~Gebert and M.~Poplavskyi.
\newblock On pure complex spectrum for truncations of random orthogonal
  matrices and {K}ac polynomials.
\newblock eprint = \texttt{arXiv:1905.03154}.

\bibitem{Gin65}
J.~Ginibre.
\newblock Statistical ensembles of complex, quaternion, and real matrices.
\newblock {\em J. Mathematical Phys.}, 6:440--449, 1965.

\bibitem{HJL15}
S.~Hameed, K.~Jain, and A.~Lakshminarayan.
\newblock Real eigenvalues of non-{G}aussian random matrices and their
  products.
\newblock {\em J. Phys. A}, 48(38):385204, 26, 2015.

\bibitem{IM18}
G.~C.~P. Innocentini and M.~Novaes.
\newblock Time-inhomogeneous random {M}arkov chains.
\newblock {\em J. Stat. Mech. Theory Exp.}, (10):103202, 14, 2018.

\bibitem{IK14}
J.~Ipsen and M.~Kieburg.
\newblock Weak commutation relations and eigenvalue statistics for products of
  rectangular random matrices.
\newblock {\em Phys. Rev. E.}, 89:032106, 2014.

\bibitem{KSZ10}
B.~A. Khoruzhenko, H.-J. Sommers, and K.~\.{Z}yczkowski.
\newblock Truncations of random orthogonal matrices.
\newblock {\em Phys. Rev. E (3)}, 82(4):040106, 4, 2010.

\bibitem{K15}
P.~Kopel.
\newblock Linear statistics of non-{H}ermitian matrices matching the real or
  complex {G}inibre ensemble to four moments.
\newblock eprint = \texttt{arXiv:1510.02987}.

\bibitem{L13}
A.~Lakshminarayan.
\newblock On the number of real eigenvalues of products of random matrices and
  an application to quantum entanglement.
\newblock {\em J. Phys. A: Math. Theor.}, 46:152003, 2013.

\bibitem{LWW18}
D.-Z. Liu, D.~Wang, and Y.~Wang.
\newblock Lyapunov exponent, universality and phase transition for products of
  random matrices.
\newblock eprint = \texttt{arXiv:1810.00433}.

\bibitem{LW19}
D.-Z. Liu and Y.~Wang.
\newblock Phase transitions for infinite products of large non-{H}ermitian
  random matrices.
\newblock eprint = \texttt{arXiv:1912.11910}.

\bibitem{LW16}
D.-Z. Liu and Y.~Wang.
\newblock Universality for products of random matrices {I}: {G}inibre and
  truncated unitary cases.
\newblock {\em Int. Math. Res. Not. IMRN}, (11):3473--3524, 2016.

\bibitem{PS18}
M.~Poplavskyi and G.~Schehr.
\newblock Exact persistence exponent for the 2d-diffusion equation and related
  {K}ac polynomials.
\newblock {\em Phys. Rev. Lett.}, 121:150601, 2018.

\bibitem{PTZ17}
M.~Poplavskyi, R.~Tribe, and O.~Zaboronski.
\newblock On the distribution of the largest real eigenvalue for the real
  {G}inibre ensemble.
\newblock {\em Ann. Appl. Probab.}, 27(3):1395--1413, 2017.

\bibitem{S17}
N.~Simm.
\newblock On the real spectrum of a product of {G}aussian matrices.
\newblock {\em Electron. Commun. Probab.}, 22:Paper No. 41, 11, 2017.

\bibitem{S17-2}
N.~J. Simm.
\newblock Central limit theorems for the real eigenvalues of large {G}aussian
  random matrices.
\newblock {\em Random Matrices Theory Appl.}, 6(1):1750002, 18, 2017.

\bibitem{SW08}
H.-J. Sommers and W.~Wieczorek.
\newblock General eigenvalue correlations for the {G}inibre ensemble.
\newblock {\em J. Phys. A: Math. Theor.}, 41(40), 2008.

\bibitem{TV15}
T.~Tao and V.~Vu.
\newblock Random matrices: universality of local spectral statistics of
  non-{H}ermitian matrices.
\newblock {\em Ann. Probab.}, 43(2):782--874, 2015.

\bibitem{TKZ12}
R.~Tribe, S.~K. Yip, and O.~Zaboronski.
\newblock One dimensional annihilating and coalescing particle systems as
  extended {P}faffian point processes.
\newblock {\em Electron. Commun. Probab.}, 17:no. 40, 7, 2012.

\bibitem{TZ11}
R.~Tribe and O.~Zaboronski.
\newblock Pfaffian formulae for one dimensional coalescing and annihilating
  systems.
\newblock {\em Electron. J. Probab.}, 16:no. 76, 2080--2103, 2011.

\end{thebibliography}

\end{document}